\documentclass[10pt]{amsart}

\makeatletter

\usepackage{amssymb, amsfonts}
\usepackage{amsmath}
\usepackage{amsthm}
\usepackage{enumerate}

\newtheorem{theorem}{Theorem}[section]
\newtheorem{lemma}[theorem]{Lemma}
\newtheorem{proposition}[theorem]{Proposition}
\newtheorem{corollary}[theorem]{Corollary}

\newcommand{\R}{\mathbb{R}}
\newcommand{\N}{\mathbb{N}}
\newcommand{\C}{\mathbb{C}}

\newcommand{\cB}{{\mathcal B}}   
\newcommand{\cC}{{\mathcal C}}   
\newcommand{\cD}{{\mathcal D}}

\newcommand{\cH}{{\mathcal H}}

\newcommand{\cL}{{\mathcal L}}

\newcommand{\cR}{{\mathcal R}}

\newcommand{\cW}{{\mathcal W}}   
   
\newcommand{\cY}{{\mathcal Y}}   
\newcommand{\cZ}{{\mathcal Z}}

\newcommand{\weak}{\rightharpoonup}

\newcommand{\embed}{\hookrightarrow}

\renewcommand{\dim}{{\rm dim}\,}
\newcommand{\eps}{\varepsilon}

\renewcommand{\Re}{\textrm{Re}}

\newcommand{\irn}{\int_{\R^N}}

\DeclareMathOperator{\spann}{span}

\DeclareMathOperator{\trace}{trace}

\begin{document}

\title{Real solutions to the nonlinear Helmholtz equation with local
nonlinearity}

\author{Gilles Evequoz}
\address{Institut f\"ur Mathematik - Goethe-Universit\"at Frankfurt, Robert-Mayer-Str.\ 10. - 60054 Frankfurt (Germany)}
\email{evequoz@math.uni-frankfurt.de}

\author{Tobias Weth}
\address{Institut f\"ur Mathematik - Goethe-Universit\"at Frankfurt, Robert-Mayer-Str.\ 10. - 60054 Frankfurt (Germany)}
\email{weth@math.uni-frankfurt.de}

\keywords{Helmholtz equation \and entire solutions \and local nonlinearity \and Dirichlet to Neumann map \and
 variational methods.}

\begin{abstract}
In this paper, we study real solutions of the nonlinear Helmholtz equation
$$
- \Delta u - k^2 u = f(x,u),\qquad x\in \R^N 
$$
satisfying the asymptotic conditions
$$
u(x)=O(|x|^{\frac{1-N}{2}}) \quad \text{and}
\quad \frac{\partial^2 u}{\partial r^2}(x)+k^2 u(x)=o(|x|^{\frac{1-N}{2}})
\qquad \text{as  $r=|x| \to \infty$.}
$$
We develop the variational framework to prove the existence of nontrivial solutions for
compactly supported nonlinearities without any symmetry assumptions.
In addition, we consider the radial case 
in which, for a larger class of nonlinearities, infinitely many
solutions are shown to exist. Our results give rise to the existence
of standing wave solutions of corresponding nonlinear Klein--Gordon
equations with arbitrarily large frequency.

\end{abstract}

\maketitle

\section{Introduction}\label{sec-introduction}
The study of the existence and qualitative properties of solutions to nonlinear wave equations
\begin{equation}\label{nlw0}
\frac{\partial^2 \psi}{\partial t^2}(t,x) - \Delta \psi(t,x) +V(x)\psi(t,x)  =f(x,\psi(t,x)), \qquad (t,x)\in\R\times\R^N,
\end{equation}
goes back to the sixties, see e.g. the classical paper by J\"orgens \cite{joergens61}. Since then,
many authors have been investigating various aspects of this problem, including the question of existence and orbital 
stability of {\em standing wave (or solitary wave) solutions}. An
ansatz for solutions of this type is given by
\begin{equation}
	\psi(t,x)=e^{-i (\omega t +\varphi)}u(x), \qquad \omega,\varphi  \in\R,
\end{equation}
with a real-valued function $u$ on $\R^N$. Taking for example a nonlinearity of the form $f(x,\psi)=g(x,|\psi|^2)\psi$ 
where $g$ is a real-valued function, we see that
such a $\psi$ solves \eqref{nlw0} if and only if $u$ solves the
reduced wave equation
\begin{equation}\label{eq:333}
- \Delta u +V(x)u - \omega^2 u = f(x,u),\qquad x\in \R^N.  
\end{equation}
The requirement that $u$ be real-valued guarantees that the
corresponding energy density 
$$
E(t,x)=\frac{1}{2}\Bigl(\bigl|\frac{\partial \psi}{\partial
  t}(t,x)\bigr|^2 +|\nabla \psi(t,x)|^2+V(x)|\psi(t,x)|^2 -\int_0^{|\psi(t,x)|^2}g(x,\tau)\,d\tau\Bigr)
$$
is constant in $t$ at every point $x\in\R^N$, a property which is characteristic of
standing wave solutions of (\ref{nlw0}). Note that, due to the presence of the linear potential $V$, it is natural to 
consider nonlinearities $f$ satisfying $\partial_u f(x,0)=0$ on $\R^N$. In this case, in
almost all of the available literature it is assumed
that $\omega^2$ is not contained in the essential spectrum of the
Schr\"odinger operator $-\Delta +V$. We refer the reader to the
surveys and monographs
\cite{ambrosetti-malchiodi-a,cerami06,rabinowitz91,struwe,stuart97,willem} and the references 
therein for results in this case. 
In the special case where $V \equiv V_0$ is
a constant, this restriction amounts to assuming $V_0>\omega^2\geq 0$. 
Some authors have also considered the limiting case where $\omega^2$ coincides with the
infimum or another boundary point of the essential spectrum of $V$, see
e.g. \cite{kuzin-pohozaev} for a survey of classical results
and \cite{alves-souto-montenegro12,ambrosetti-malchiodi-felli:05,ambrosetti-malchiodi-a,bartsch-ding99,costa-tehrani01,moroz-van-schaftingen10,schneider:03,willem-zou03}
for more recent work in this case. On the contrary, very
little seems to be known if $\omega^2$ is contained in the interior of
the essential
spectrum of $V$. 
In the present paper, we consider this situation in the special case $V \equiv V_0$ and $\omega^2>V_0$, so by setting
$k^2=\omega^2-V_0$ we arrive at the 
nonlinear Helmholtz equation 
\begin{equation}\label{eq:33}
- \Delta u - k^2 u = f(x,u),\qquad x\in \R^N,  
\end{equation}
with $k>0$. It seems unclear a priori in which space one
should approach this problem and whether variational methods can be used. 
In the present paper, we provide first results in the case where $f$
is supported in a bounded subset of $\R^N$, i.e. $f$ vanishes in $[\R^N \setminus B_R] \times
\R$ for $R>0$ sufficiently large. Here $B_R \subset \R^N$ denotes the open ball
centered at $0$ with radius $R$. Note that in this case no nontrivial solution
of \eqref{eq:33} exists in the space $L^2(\R^N)$, as follows
immediately from a classical result of Rellich, see \cite[Satz
1]{rellich43}. We will focus instead on solutions satisfying the asymptotic
conditions
\begin{equation}\label{eqn:radiation2}
u(x)= O(|x|^{\frac{1-N}{2}}) \quad \text{and}\quad 
\frac{\partial^2 u}{\partial r^2}(x) +k^2 u(x)=o(|x|^{\frac{1-N}{2}})
\quad \text{as }r=|x|\to\infty.
\end{equation}
Thus the solutions decay to zero as $|x| \to \infty$ if $N \ge 2$. 
We emphasize that in general not all solutions of (\ref{eq:33}) satisfy
(\ref{eqn:radiation2}) if $N \ge 2$. In particular, in case $N \ge 2$ and $f \equiv 0$,
(\ref{eqn:radiation2}) is not satisfied by one-dimensional standing wave
solutions of (\ref{eq:33}) given by 
\begin{equation}
  \label{eq:31}
x \mapsto \sin (kx \cdot \xi +\varphi) \qquad \text{with a unit vector $\xi \in \R^N$ and $\varphi \in \R$.}
\end{equation}

The above restriction on $f$ allows us to work with the Dirichlet to Neumann map associated with the exterior problem for the linear
Helmholtz equation $\Delta u +k^2 u=0$ on $\R^N \setminus B_R$ together
with a suitable asymptotic condition on $u$. To explain this in more detail, let us suppose for a moment that the nonlinearity $f(x,u)$ is replaced by an
inhomogeneous source term $f(x)$ supported in $B_R$. In this case, 
a well-studied problem is to analyze the far field expansion of the
(unique) complex solution of
\eqref{eq:33} satisfying the 
Sommerfeld (outgoing) radiation condition
\begin{equation}\label{eqn:radiation1}
\left|\frac{\partial u}{\partial r}(x)- ik u(x)\right|=o(r^{\frac{1-N}{2}}) \qquad \text{ as }r=|x|\to\infty.
\end{equation}
This condition has been introduced in Sommerfeld's classical work
\cite{sommerfeld12}, and it corresponds to the study of 
outgoing waves excited by
the source term $f(x)$. Moreover, by a well-known result (going
back to Rellich \cite[p.58]{rellich43}), for given and
 sufficiently regular Dirichlet boundary data on $S_R:= \partial B_R$ there exists a unique
complex solution of $\Delta u + k^2 u=0$ in $\R^N \setminus
B_R$ satisfying \eqref{eqn:radiation1}. Furthermore, the corresponding
Dirichlet to Neumann map $T_R$ on $S_R$, also called the {\em capacity
  operator} (see \cite{nedelec}), is well understood and can be
computed explicitly in terms of spherical harmonics, see
Section~\ref{sec-prop-capac-oper} below. This operator
assigns to a given 
boundary datum on $S_R$ the normal derivative of the corresponding
unique solution $\Delta u + k^2 u=0$ in $\R^N \setminus
B_R$ satisfying \eqref{eqn:radiation1}. However, condition
(\ref{eqn:radiation1}) rules out (real-valued) standing wave solutions
$u$ which are the subject of the present paper. Nevertheless, a
standing wave solution $u$ can ---
as explained in \cite[pp.~328--329]{sommerfeld12} --- be realized
as a superposition of an outgoing
wave and an incoming wave having opposite frequency and thus as real
part of a function satisfying (\ref{eqn:radiation1}). In particular, it still
satisfies the weaker asymptotic condition (\ref{eqn:radiation2}). Technically,
this amounts to working with the real part $K_R$ of
the operator $T_R$ restricted to a
Sobolev space of real functions on $S_R$, see
Section~\ref{sec-some-tools} below. With the help of the operator
$K_R$, we will be able reduce the problem of finding real solutions of 
\eqref{eq:33} satisfying \eqref{eqn:radiation2} to a variational
problem in $H^1(B_R)$. While such a reduction
 clearly requires that $f$ vanishes outside
of $B_R$, we shall also see that the vanishing of $f$
leads to difficulties in the proof
of Cerami's condition which is needed to show the existence of
critical points of the corresponding functional. 

To state our results, we need to introduce further notation and
to state our assumptions. Let $2^*$ denote the
critical Sobolev exponent, i.e. $2^*:=2N/(N-2)$ if $N\ge 3$ and
$2^*:=\infty$ if $N=1,2$. For our main result, we shall suppose that the nonlinearity $f: \R^N \times \R \to \R$ is
continuous, and that there exists a bounded set 
$\Omega\subset \R^N$ of positive measure such that the following holds: 
\begin{itemize}
\item[$(f_0)$] $f(x,u)=0$ for $x \in \R^N \setminus \Omega$, $u \in \R$.
\item[$(f_1)$] There exists $a>0$, $p\in(2,2^*)$ such that $|f(x,u)| \le a(1+|u|^{p-1})$
for every $x \in \R^N$, $u \in \R$.
\item[$(f_2)$] $f(x,u)=o(|u|)$ uniformly in $x$ as $u \to 0$.
\item[$(f_3)$] $F(x,u) \ge 0$ for every $x \in \R^N$, $u \in \R$, and $F(x,u)/u^2 \to
  \infty$ as $|u| \to\infty$ for every $x \in \Omega$. 
\item[$(f_4)$] There exists $s_0>0$ such that for every $x \in \Omega$
  we have $f(x,-s_0) \le 0\le f(x,s_0)$, and the map $u \mapsto f(x,u)/|u|$ is
  nondecreasing on $(s_0,\infty)$ and on $(-\infty,-s_0)$. 
\end{itemize}

Here we set $F(x,u)=\int_0^u f(x,s)\,ds$ for $x \in \R^N$, $u \in \R$. 
We point out that nonlinearities of the type 
$$
f(x,u)=q(x)|u|^{p-2}u\qquad \text{or}\qquad 
f(x,u)= q(x)u\log(1+|u|^s), \quad s>0,
$$ 
satisfy these assumptions if $q$ is
continuous and $q>0$ on $\Omega$, $q \equiv 0$ on $\R^N \setminus
\Omega$. Moreover, if $f$ satisfies $(f_0)$--$(f_4)$ and if $g:\R^N
\times \R \to [0,\infty)$ is continuous, vanishes outside of a 
bounded subset of $\Omega\times[0,\infty)$ and satisfies $(f_2)$, 
then the sum $f+g$ also satisfies $(f_0)$--$(f_4)$.

We need one more definition related to the
asymptotics of the solutions we consider. For $R>0$, we
let $\cR_R$ denote the set of all functions $u \in H^1_{loc}(\R^N,\R)$
such that $u= \Re (w)$ in $\R^N \setminus B_R$ for a function $w \in 
\cC(\R^N \setminus B_R,\C) \cap \cC^\infty(\R^N \setminus
\overline{B_R},\C)$ satisfying\\[0.2cm]
(i) $w \equiv u$ on $S_R$,\\
(ii) $\Delta w + k^2 w = 0$ in $\R^N \setminus \overline{B_R}$,\\
(iii) $w(x)= O(r^{\frac{1-N}{2}})$ and $\bigl|\frac{\partial w}{\partial
    r}(x)- ik w(x)\bigr|=o(r^{\frac{1-N}{2}})$ as $r=|x|\to\infty.$\\[0.2cm]
Note that every function $u \in \cR_R$ satisfies the asymptotic
conditions (\ref{eqn:radiation2}). Our main result is the following. 

\begin{theorem}\label{thm:1}
There exists a countable set $\cD \subset (0,\infty)$
such that for $R \in (0,\infty) \setminus \cD$ with $\Omega \subset
B_R$ and every nonlinearity $f$ satisfying assumptions $(f_0)$ --
$(f_4)$ we have:\\ 
(i) Equation \eqref{eq:33} admits
a solution in $\cR_R$.\\
(ii) If, in addition, $f(x,-t)=-f(x,t)$ for all $(x,t)\in\R^N\times\R$, then
equation \eqref{eq:33} admits a sequence of
pairs of solutions $\{\pm u_n\}$, $n \in \N$ in $\cR_R$ with the property that
\begin{equation}
  \label{eq:28}
\|u_n\|_{H^1(B_R)} \to \infty \qquad \text{as $n \to \infty$.}  
\end{equation}
\end{theorem}

Some remarks seem in order. Elliptic regularity theory implies that 
the solutions given by Theorem~\ref{thm:1} 
belong to $W^{2,q}_{loc}(\R^N)$ for all $1\leq q<\infty$, so there
are strong solutions of \eqref{eq:33} in 
$\cC^{1,\alpha}_{\text{loc}}(\R^N)$ for all $0<\alpha<1$. The set
$\cD$ in Theorem~\ref{thm:1} is defined by the property that an
associated $R$-dependent linear eigenvalue problem admits the eigenvalue
zero iff $R \in \cD$, see Lemma~\ref{sec:some-tools-2}(iv) below for details. The
restriction $R \in (0,\infty)\setminus \cD$ in Theorem~\ref{thm:1} is
not necessary, but the proof is more complicated in the case where
$R \in \cD$. Since the choice of $R$ has no impact on the validity of
the asymptotic conditions (\ref{eqn:radiation2}), we decided not to
consider the case $R\in \cD$ in the present paper. Nevertheless, it is natural to ask whether different choices of $R$ in
Theorem~\ref{thm:1} give rise to different solutions of
problem (\ref{eq:33}),~(\ref{eqn:radiation2}). The following partial answer to this question can be derived from a
careful study of the explicit representation of the
capacity operator (and its real part) discussed in
Section~\ref{sec-prop-capac-oper} below. 

\begin{theorem}
\label{thm:1-1}
For fixed $R>0$ and $u \in \cR_R$ with $u|_{S_R} \not \equiv 0$, there exist at most
countably many $R'>R$ such that $u \in \cR_{R'}$.\\
As a consequence, the problem \eqref{eq:33},~(\ref{eqn:radiation2}) admits infinitely
many solutions if the nonlinearity $f$ satisfies assumptions $(f_0)$ --
$(f_4)$.
\end{theorem}

It might be somewhat surprising at first glance that Theorem~\ref{thm:1-1} also gives rise to
infinitely many solutions $u_n,\:n \in \N$ of \eqref{eq:33},~\eqref{eqn:radiation2}
in the case where $f$ is {\em not} odd in its second variable. However, a local 
unboundedness property as in \eqref{eq:28} should
not be expected for these solutions.  
It is natural to ask whether it is possible to relax the condition
that $f$ vanishes outside a compact set. We have no general answer to this
question yet, but in the case where
$f$ is radial in $x$, a shooting argument yields radial solutions of
\eqref{eq:33},~\eqref{eqn:radiation2} under much less restrictive
assumptions on $f$. For a precise result, see
Theorem~\ref{thm:radial} below. We do not even need to assume that
$f$ tends to $0$ as $r=|x| \to \infty$. 

The paper is organized as follows. In Section~\ref{sec-some-tools}, we
set up the variational framework used to prove Theorem~\ref{thm:1}. 
Here we also state key properties of the capacity operator $T_R$ and its real
part $K_R$ in the case $N \ge 2$, but we postpone the derivation of these properties to
Section~\ref{sec-prop-capac-oper} since the underlying computations ---
relying on special properties of Hankel functions --- 
are somewhat technical. 
In Section~\ref{sec-thm-1} we then complete the proof of Theorem~\ref{thm:1} in
the case $N \ge 2$. As
already remarked above, a key difficulty in the proof is the validity of Cerami's
condition (see Proposition~\ref{prop:cerami}), and the proof of this
property is rather long. In Section~\ref{sec-thm-1}, we also establish
--- under stronger assumptions on the
nonlinearity --- a rigid minimax principle
with respect to families of half-spaces for the solution which minimizes
the corresponding energy functional among all critical points, see
Theorem~\ref{sec:proof-theor-refthm:1-2}. We believe that this
minimax principle gives rise to additional properties of the
corresponding (ground state) solutions. Section~\ref{sec-remarks} contains a sketch of the proof of
Theorem~\ref{thm:1} in the one-dimensional case. This case is
much easier than the case $N \ge 2$ but has to be treated slightly
differently. Section~\ref{sec-radial} is devoted to the
radial case. As pointed out already, we apply a shooting argument
to prove the existence of infinitely many radial solutions of
\eqref{eq:33},~\eqref{eqn:radiation2} under much less restrictive
assumptions on $f$, see Theorem~\ref{thm:radial}.
Finally, as noted above, in Section~\ref{sec-prop-capac-oper} 
we derive key properties of the capacity operator $T_R$ and its real
part $K_R$. We note that some of these properties are well
known (see e.g. \cite{nedelec} for the case $N=3$), but we 
could not find an appropriate reference for general $N \ge
2$. Moreover, it seems that the operator $K_R$ has not
been studied in the degree of detail which we need for our
purposes. The final part of Section~\ref{sec-prop-capac-oper} is
devoted to the proof of Theorem~\ref{thm:1-1}.

We close this introduction with a remark on some open
questions which we plan to address in future work. First, one may ask for which compactly supported nonlinearities $f$ satisfying
$(f_0)$ there exist solutions of (\ref{eq:33}) given as superposition
of a function of type (\ref{eq:31}) and a function
satisfying (\ref{eqn:radiation2}). Second, one may try to prove the existence of standing
wave solutions without assumption $(f_0)$. This would require a
completely different (variational) approach. Another natural question,
more closely related to the scattering theory of (\ref{nlw0}), is
the existence of complex solutions of (\ref{eq:33}) satisfying the radiation
condition (\ref{eqn:radiation1}). It is not clear whether solutions can be
expected in the case where $f(\cdot,0)\equiv 0$ as assumed in $(f_2)$. This problem is non-variational due to the lack of symmetry
of the Dirichlet to Neumann map $T_R$ associated with
(\ref{eqn:radiation1}), so a completely different approach has to be used.

\section{The variational framework}\label{sec-some-tools}
We assume that $N \ge 2$ in the next
two sections, referring to Section~\ref{sec-remarks} for the case
$N=1$. In this section we will introduce the capacity operator
and  develop a variational framework for problem~\eqref{eq:33},~\eqref{eqn:radiation2}. 
We start by fixing some notation.  Let $R>0$ be such that $\Omega \subset B_R$, where
$B_R:=B_R(0)$ is the open ball with radius $R$ centered at
zero. We also set $E_R:= \R^N \setminus \overline{B_R}$ and consider the space  
\begin{align}
H_R^{\C}:= \Bigl\{u \in H^1_{loc}(E_R,\C)\::\: &\frac{u}{(1+r^2)^{1/2}} \in
  L^2(E_R,\C),\: \frac{\nabla u}{(1+r^2)^{1/2}} \in
  L^2(E_R,\C^N), \nonumber\\
&\frac{\partial u}{\partial r}  -i k u \in
  L^2(E_R,\C)  \label{eq:1}
\Bigr\},
\end{align}
where --- here and in the following --- $r$ always denotes the radial
variable, i.e., $r=|x|$. It is known (see \cite{nedelec} for the case
$N=3$ and Section~\ref{sec-prop-capac-oper} below for general $N \ge 2$) 
that for every $u \in H^{\frac12}(S_R,\C)$, $S_R:=\partial B_R$, there
exists a unique weak solution $w \in H_R^\C$ of the problem
\begin{equation}
  \label{eq:2}
\left\{
  \begin{aligned}
\Delta w + k^2 w &= 0 &&\qquad \text{in $E_R$,}\\
w&=u &&\qquad \text{on $S_R$}.  
  \end{aligned}
\right.
\end{equation}
Here weak solution means that $\trace(w)=u$ and 
$$
\int_{E_R} \Bigl(\nabla w\cdot \nabla \phi-k^2 w \phi\Bigr) \,dx=0 \qquad
\text{for all $\phi \in \cC^1_c(E_R)$.}
$$
where $\cC^1_c(E_R)$ denotes the space of $\cC^1$-functions with
compact support in $E_R$. We then define the {\em capacity operator} 
(or {\em Dirichlet to Neumann map}) 
\begin{equation}
  \label{eq:17}
T_R: H^{\frac12}(S_R,\C) \to H^{-\frac12}(S_R,\C), \qquad T_R u= \frac{\partial w}{\partial \eta} \in H^{-\frac12}(S_R,\C)
\end{equation}
where $w \in H_R^\C$ is the unique solution of \eqref{eq:2}
corresponding to $u \in H^{\frac12}(S_R,\C)$ and $\eta$ denotes the normal unit vector field on $S_R$ pointing outside 
$B_R$ and inside $E_R$.\\
As shown in \cite{nedelec} for the case $N=3$ and detailed in
Section~\ref{sec-prop-capac-oper} below
for general $N \ge 2$, the operator $T_R:H^{\frac12}(S_R,\C) \to
H^{-\frac12}(S_R,\C)$ is continuous, and suitably normalized spherical harmonics (when
considered as functions of spherical angles) form an 
orthonormal basis of eigenfunctions of
$T_R$. The operator $T_R$ is not symmetric and therefore does not give
rise to a variational formulation of our nonlinear problem given by
\eqref{eq:33} and \eqref{eqn:radiation2}. Therefore, we set 
$H^{\frac12}(S_R)=H^{\frac12}(S_R,\R)$,
$H^{-\frac12}(S_R)=H^{-\frac12}(S_R,\R)$, and we let
\begin{equation}
  \label{eq:27}
K_R: H^{\frac12}(S_R) \to H^{-\frac12}(S_R),\qquad K_R u = \Re [T_R u]
\end{equation}
denote the real part of the restriction of $T_R$ to
$H^{\frac12}(S_R)$. This operator can be seen as the Dirichlet to
Neumann map corresponding to the problem \eqref{eq:2} with real data
$u$ on $S_R$ and solutions $w$ given as a real part of a function 
in $H_R^{\C}$. The explicit calculations of $w$ in terms of $u$ in
Section~\ref{sec-prop-capac-oper} below immediately imply that $w$ satisfies the asymptotic
conditions \eqref{eqn:radiation2}. Moreover, the operator $K_R$ has
the following properties.

\begin{lemma}
\label{sec:capacity-operator}
The operator $K_R$ is bounded, symmetric and negative definite. More
precisely, there are constants $\gamma_R>0$ and $\Gamma_R>0$ such that
\begin{equation}
  \label{eq:5}
\|K_R u\|_{H^{-\frac{1}{2}}(S_R)} \le \Gamma_R
\|u\|_{H^{\frac{1}{2}}(S_R)}\quad \text{and}\quad  
\int_{S_R} u K_R u  \,d \sigma \le - \gamma_R \int_{S_R}u^2 \,d\sigma
\end{equation}
for all $u \in H^{\frac12}(S_R)$. Moreover,
\begin{equation}
  \label{eq:16}
\int_{S_R} v K_R u  \,d \sigma= \int_{S_R} u K_R  v  \,d \sigma
\qquad \text{for $u,v \in H^{\frac12}(S_R)$.}
\end{equation}
\end{lemma}

We postpone the proof of this lemma to Section~\ref{sec-prop-capac-oper} below. 
Setting $X:=H^1(B_R)$, we now recall the standard estimate  
\begin{equation}
  \label{eq:30}
\int_{B_R}u^2\,dx \le c \int_{B_R} |\nabla u|^2\,dx
+c \int_{S_R}u^2\,d\sigma  \qquad \text{for $u \in X$}
\end{equation}
with some constant $c=c(R)>0$, see e.g. \cite[Theorem A.9]{struwe}.
Moreover, we consider the bilinear form 
\begin{equation}
  \label{eq:25}
\cB_k: X\times X \to \R,\qquad \cB_k(u,v) = \int_{B_R}
\Bigl(\nabla u\cdot\nabla v-k^2 uv\Bigr)\,dx
-\int_{S_R} v K_R u \,d \sigma.
\end{equation}
From \eqref{eq:5} and (\ref{eq:30}), we easily deduce 
the following

\begin{corollary}
\label{sec:some-tools-1}
$\cB_0$ defines a scalar product on $X=H^1(B_R)$ which 
is equivalent to the standard scalar product, i.e. the
corresponding norms are equivalent. 
\end{corollary}

We also set $\cB_k(u):= \cB_k(u,u)$ for $u \in X$ in the
following. In the next lemma, we collect key facts concerning
$\cB_k$ and the (nonlocal) eigenvalue problem 
\begin{equation}
  \label{eq:10}
 \left\{ \begin{aligned}
 -\Delta u &= \lambda u,&&\qquad \text{in $B_R$,}\\ 
\frac{\partial u}{\partial \eta} &= K_R u &&\qquad \text{on $S_R$.}      
  \end{aligned}
\right. 
\end{equation}

\begin{lemma}
\label{sec:some-tools-2}
(i) The eigenvalue problem \eqref{eq:10} admits an unbounded sequence
of eigenvalues $0<\lambda_1 \le \lambda_2 \le ...$ and a corresponding
system of eigenfunctions $e_j$, $j \in \N$ which is complete in
$X$. Moreover, each eigenfunction $e_j$, $j \in \N$ is analytic in $B_R$.\\
(ii) There exists a scalar product $\langle \cdot, \cdot \rangle$,
equivalent to the standard scalar product on $X=H^1(B_R)$, and an orthogonal
splitting $X=X^-\oplus X^0 \oplus X^+$ such that 
$$
\cB_k(u)  = \|u^+\|^2 -\|u^-\|^2\quad \text{ for all }u\in X,
$$
where $\|\cdot\|=\langle \cdot, \cdot \rangle^{\frac12}$ is the
induced norm, and $u^\pm, u^0$ denote the corresponding orthogonal projections of $u$ onto $X^\pm$, $X^0$,
respectively. More precisely,
\begin{equation*}
X^- = \spann \{e_j\,: j \le j_*\}, X^0= \spann \{e_j\,:
j_* <j < j^*\} \text{ and }  X^+ = \overline{\spann
  \{e_{j}\,: j \ge j^*\}}, 
\end{equation*}
where $j_* := \max \{j \in \N\::\:
\lambda_j < k^2\}$ and $j^*:= \min \{j \in
\N\::\: \lambda_j>k^2 \}$. In particular, $X^-$ and $X^0$ are finite dimensional.\\
(iii) The family $e_j$, $j \in \N$ is orthogonal with respect to the
scalar products $\langle \cdot,\cdot \rangle$, $\cB_0$ and
the scalar product of $L^2(B_R)$.\\
(iv) There exists a countable set $\cD \subset (0,\infty)$ such that
 $X^0 \not= \{0\}$ if and only if $R  \in \cD$.    
\end{lemma}

\begin{proof}
(i) We first consider $H^1(B_R)$ with the equivalent scalar product
$\cB_0$. Since $K_R$ is negative definite, all eigenvalues of 
\eqref{eq:10} must be positive. Hence $u \in H^1(B_R)$ is an eigenfunction of
\eqref{eq:10} corresponding to $\lambda \in \R$ if and only if
$\lambda>0$ and $u$ is an eigenfunction of the operator $K \in \cL(X)$
defined by 
\begin{equation}\label{equation:K}
\cB_0(Ku,v)= \int_{B_R}uv\,dx \qquad \text{for $u,v \in X$}
\end{equation}
corresponding to the eigenvalue $\frac{1}{\lambda}$.
The operator $K$ is bounded, symmetric with respect to $\cB_0$, nonnegative and compact, since the
embedding $X \hookrightarrow L^2(B_R)$ is compact. Moreover, $0$ is
not an eigenvalue of $K$. Hence $K$ admits a sequence of positive eigenvalues
$\mu_1 \ge \mu_2 \ge ...$ such that $\mu_j \to 0$ as $j \to \infty$,
and a corresponding complete system of $\cB_0$-orthogonal
eigenfunctions $e_j$, $j \in \N$. Assertion (i) now follows with
$\lambda_j= \frac{1}{\mu_j}$, $j \in \N$ and the family $\{e_j\::\: j
\in \N\}$ thus obtained.\\
(ii) For $u \in X$, let $u^\pm, u^0$ denote the $L^2(B_R)$-orthogonal 
projections of $u$ onto the subspaces $X^\pm$, $X^0$, respectively,
as defined in the assertion. For $u,v \in X$,
we define
$$
\langle u,v \rangle = \cB_k(u^+,v^+)- \cB_k(u^-,v^-) +
\int_{B_R} u^0 v^0\,dx.
$$
It is easy to see that this scalar product has the desired
properties, and by construction the splitting $X=X^-\oplus X^0 \oplus
X^+$ is orthogonal with respect to this scalar product.\\ 
(iii) The $\cB_0$-orthogonality of the family $\{e_j\::\:j \in \N\}$ has
already been shown above, and the $L^2(B_R)$-orthogonality then follows from
\eqref{equation:K}. From this, the orthogonality with respect to $\langle
\cdot, \cdot \rangle$ immediately follows by definition.\\
(iv) This part, which relies on special properties of the capacity
operator and Hankel functions, will be proved in
Section~\ref{sec-prop-capac-oper} below.
\end{proof}

We now consider the functional 
\begin{equation}\label{eqn:Phi}
\Phi: X \to \R,\qquad \Phi(u) = \frac{1}{2} \cB_k(u) - \varphi(u)= \frac{1}{2}\Bigl(\|u^+\|^2-\|u^-\|^2\Bigr)-\varphi(u),
\end{equation}
where $\varphi(u)= \int_{B_R}F(x,u(x))\,dx$ for
$u \in X$ and $F(x,t)= \int_0^t f(x,s)\,ds$ for $t \in \R$. 
It is well known that $\varphi \in \cC^1(X,\R)$ as a consequence of assumption
$(f_1)$, and 
\begin{equation}
  \label{eq:phi}
\text{$\varphi$ is nonnegative on $X$}  
\end{equation}
by $(f_3)$. Moreover, the critical points of
$\Phi$ correspond to restrictions to $B_R$ of solutions of \eqref{eq:33}. Indeed, 
if $u\in X$ is a critical point of  $\Phi$, then 
$$
0 = \int_{B_R} \Bigl(\nabla u\cdot \nabla w
-k^2 u w - f(x,u) w \Bigr)\,dx - \int_{S_R}w K_R u\,d \sigma
$$
for every $w \in X$, hence $u$ is a weak solution of the
problem 
\begin{equation}
  \label{eq:33_i}
\left \{
  \begin{aligned}
- \Delta u - k^2 u &= f(x,u)&&\qquad  \text{in $B_R$},\\
\frac{\partial u}{\partial\eta} &= K_R u &&\qquad \text{on $S_R$}.\\       
  \end{aligned}
\right.
\end{equation}
As explained in Section~\ref{sec-prop-capac-oper}, extending $u$ by the
real part of the unique solution of \eqref{eq:2} in $H_R^{\C}$ then yields a
solution of \eqref{eq:33},~\eqref{eqn:radiation2}.

\section{Proof of Theorem 1.1}\label{sec-thm-1}
We fix $R \in (0,\infty) \setminus \cD$ from now on, so that $X^0=
\{0\}$ by Lemma~\ref{sec:some-tools-2}. As remarked in the
introduction, this restriction is made only in order to simplify the proofs, whereas all
results can still be proved --- with additional effort --- in the case where
$X^0$ has positive dimension.

We first collect useful facts about the functional $\Phi$ defined in
\eqref{eqn:Phi}. For this we recall the following well-known consequence of ($f_1$) and ($f_2$):
\begin{equation}
  \label{eq:standardestimate}
\forall\,\eps>0,\, \exists\, C_\eps>0 \text{ such that } |f(x,u)|\leq \eps|u|+C_\eps |u|^{p-1},
\text{ }\forall\;(x,u)\in\R^N\times\R.
\end{equation}

\begin{lemma}\label{lemma-new1}
\begin{enumerate}
\item[(i)] There exists $\alpha_0>0$ such that $\inf
  \limits_{\Sigma_\alpha}\Phi>0$ for $\alpha \in (0,\alpha_0)$,
where $\Sigma_\alpha:= \{u \in X^+\,:\,\|u\|=\alpha\}$.
\item[(ii)] Let $\cZ$ be a closed cone contained in a finite-dimensional 
subspace $W$ of $X$ and such that 
  \begin{equation}
    \label{eq:19}
    \begin{aligned}
&\text{$\{x \in \Omega\::\: w(x) \not=0\}$ has positive measure for}\\
&\text{every $w \in \cZ \setminus \{0\}$ with $\|w^+\| \ge \|w^-\|$.} 
    \end{aligned}
  \end{equation}
Then there exists $\rho=\rho(\cZ)>0$ such that $\Phi(u)\leq 0$ for all $u\in \cZ$ 
satisfying $\|u\|\geq \rho$.
\end{enumerate}
\end{lemma}
Here and in the following, a set $\cZ \subset X$ is called a 
cone if $\lambda x \in \cZ$ for every $x \in
\cZ$, $\lambda \ge 0$. In particular, (ii) applies to $\cZ= W$ if $W$
is a finite-dimensional subspace of $X$.

\begin{proof}
(i) For $u \in X^+$ we have $\Phi(u) = \frac{1}{2}\|u\|^2-\varphi(u)$ and
$\varphi(u) = o(\|u\|^2)$ as $u \to 0$ by \eqref{eq:standardestimate} and
Sobolev embeddings. Hence the conclusion
follows.\\
(ii) Suppose by contradiction that a sequence $(u_n)_n\subset \cZ$ exists with $\Phi(u_n)> 0$ for all
$n$ and $\|u_n\|\to\infty$ as $n\to\infty$. Setting
$w_n=\frac{u_n}{\|u_n\|}$, we may pass to a subsequence such that $w_n
\to w \in W$ since $W$ is finite-dimensional. Since $w_n \in \cZ$ and
$\|w_n\|=1$ for all $n$, we have $w \in \cZ$ and $\|w\|=1$. Moreover, by
\eqref{eq:phi} we have 
$$
0\le \liminf_{n \to \infty} \frac{\Phi(u_n)}{\|u_n\|^2}\le \frac12
\lim_{n \to \infty} (\|w_n^+\|^2-\|w_n^-\|^2)
= \frac12 ( \|w^+\|^2-\|w^-\|^2)
$$
and therefore $\|w^+\| \ge \|w^-\|$. Hence
\eqref{eq:19} implies that $\Omega_w:=\{x \in \Omega\::\: w(x)
\not=0\}$ has positive measure. Passing to a subsequence, we may also
assume that $w_n \to w$ pointwise a.e. in $B_R$, which implies that 
$$
|u_n(x)|=\|u_n\| |w_n(x)|\to\infty \qquad \text{as $n\to\infty$ for
  a.e. $x \in \Omega_w$.}
$$
By $(f_3)$ and Fatou's Lemma, we therefore deduce that
$$
0\le \frac{\Phi(u_n)}{\|u_n\|^2}\le \frac12 (\|w_n^+\|^2-\|w_n^-\|^2) 
- \int_{\Omega_w}\frac{F(x,u_n)}{u_n^2}w_n^2\, dx\to-\infty
$$
as $n\to\infty$. This contradiction proves the claim.
\end{proof}

To proceed with the proof of Theorem~\ref{thm:1}, we shall decompose the nonlinearity $f$ 
as follows. We write $f=f_1+f_2$, where 
$f_i: \R^N \times \R \to \R$, $i=1,2$ are defined by 
$$
f_1(x,u)= 
\left \{
  \begin{aligned}
& f(x,u),&&|u| \ge s_0,\\
& \frac{u^2}{s_0^2 }f(x,s_0),&&0 \le u \le s_0,\\ 
& \frac{u^2}{s_0^2}f(x,-s_0),&&-s_0 \le u \le 0,
  \end{aligned}
\right.
$$
and we put $f_2= f-f_1$. Setting
$F_i(x,u)=\int_0^u f_i(x,s)\,ds$, $i=1,2$ for $x \in \R^N$, $u \in
\R$, we see that 
\begin{equation}
  \label{eq:6}
\begin{aligned}
&\text{$f_2(x,u)=0$ if $x \in \R^N \setminus
\Omega$ or $|u| \ge s_0$;}\\ 
&\text{$F_2(x,u)=0$ if $x \in \R^N \setminus
\Omega$, $u \in \R$;}\\ 
&\text{$f_2$ and $F_2$ are bounded on $\R^N \times \R$.}   
\end{aligned}
\end{equation}
Moreover, $f_1$ satisfies condition $(f_2)$ and the following stronger version of condition
$(f_4)$. For every $x \in \R^N$,
\begin{equation}
  \label{eq:13}
\text{$u \mapsto \frac{f_1(x,u)}{|u|}$ is nondecreasing on
  $\R \setminus \{0\}$.}    
\end{equation}

We decompose the functional $\varphi: X \to \R$ accordingly and write $\varphi=\varphi_1 + \varphi_2$ with 
$$
\varphi_i(u) = \int_{B_R} F_i(x,u(x))\,dx = \int_{\Omega}
F_i(x,u(x))\,dx, \qquad i=1,2.
$$
We note that $\varphi_2$ is bounded on $X$ by \eqref{eq:6}. The following proposition will 
be a main technical step in the proof
of Theorem~\ref{thm:1}.

\begin{proposition}
\label{prop:cerami}
$\Phi$ satisfies the Cerami condition in $X$,  i.e., every
sequence $(u_n)_n\subset X$ such that $\Phi(u_n) \to c$ for some $c
\in \R$ and $(1+\|u_n\|)\|\Phi'(u_n)\| \to 0$ as $n\to\infty$ has a 
subsequence which converges in $X$.
\end{proposition}

The proof of this proposition is quite long and requires subtle
estimates. Parts of the proof are inspired by \cite{li-wang11} and \cite{liu12},
but we need new arguments to deal with the difficulty that the
nonlinearity may vanish on a subset of $B_R$ of positive measure. A key role in the
proof is played by the useful inequality 
\begin{equation}
  \label{eq:14}
  \begin{aligned}
  &f_1(x,u)[s(\frac{s}{2}+1)u+(1+s)v]+F_1(x,u)-F_1(x,[1+s]u+v)\le 0\\
&\text{for $x \in \R^N$, $u,v \in \R$ and $s\ge -1$,}
  \end{aligned}
\end{equation}
which follows from (\ref{eq:13}). Indeed, as noted in \cite{liu12}, this
inequality is a weak version of \cite[Lemma 2.2]{szulkin-weth09}. As a
consequence of (\ref{eq:14}) and of the properties of $f_2$, we may derive
the following 

\begin{lemma}
  \label{lem1extra02}
For every $K>0$ there is a constant $C=C(K)>0$ with the following
property. If $u,s,v \in \R$ are numbers with $-1 \le s \le K$ and $|v|
\le K$, then 
$$
f(x,u)[s(\frac{s}{2}+1)u+(1+s)v]+F(x,u)-F(x,[1+s]u+v)\le C \qquad
\text{for all $x \in \R^N$.}
$$
\end{lemma}

\begin{proof}
By \eqref{eq:6} there exists a constant $C_1>0$ (depending
on $K$) such that 
$$
|f_2(x,u)| \le C_1, \quad |F_2(x,u)| \le C_1\quad \text{and} \quad |f_2(x,u)s(\frac{s}{2}+1)u| \le C_1
$$
for $u \in \R$, $x \in \R^N$, $|s| \le K+1$. Consequently, we have 
\begin{align*}
f_2(x,u)&[s(\frac{s}{2}+1)u+(1+s)v]+F_2(x,u)-F_2(x,[1+s]u+v) \\
&\le C_1[1+K(K+1)] + 2 C_1 
\end{align*}
for $x \in \R^N$ and $u,s,v \in \R$ with $-1 \le s \le K$ and $|v| \le K$. 
Since $f=f_1+f_2$, $F=F_1+F_2$ and (\ref{eq:14}) holds, the claim follows with $C:=C_1[1+K(K+1)] + 2 C_1$.
\end{proof}

The next step in the proof of Proposition~\ref{prop:cerami} is the
following relative energy estimate. In the following, a sequence $(u_n)_n$ in $X$ is called a
{\em Cerami sequence} for $\Phi$ if $\Phi(u_n) \to c$ for some $c
\in \R$ and $(1+\|u_n\|)\|\Phi'(u_n)\| \to 0$ as $n\to\infty$.

\begin{lemma}
\label{sec:proof-theor-refthm:1-1}
For every $\kappa>0$ there exists $\tilde C= \tilde C(\kappa)>0$ with the following property. If
$(u_n)_n$ is a Cerami sequence for $\Phi$, and $r_n \ge 0$, $v_n \in X^-$, $n \in \N$ are given with $r_n \le
\kappa$ and $\|v_n\| \le \kappa$ for all $n \in \N$, then 
\begin{equation*}
\Phi(r_n u_n+v_n) \le \Phi(u_n) + \tilde C +o(1) \qquad \text{as $n \to \infty$.}
\end{equation*}
\end{lemma}

\begin{proof}
We first note that, by a standard bootstrap argument using elliptic
regularity theory, there exists a
constant $K=K(\kappa)>\kappa$ such that 
$$
\|v\|_{L^\infty(B_R)} \le K \qquad \text{for every $v \in X^-$ with $\|v\| \le \kappa.$}
$$
We write $r_n = 1+s_n$ with $-1 \le s_n \le
\kappa-1 \le K$ and set $w_n=r_n u_n+v_n= (1+s_n) u_n +v_n$ for $n \in
\N$. Then 
\begin{align*}
&\Phi (w_n)-\Phi(u_n) = \frac{1}{2}[\cB_k(w_n)-\cB_k(u_n)]+\int_{B_R} (F(x,u_n)-F(x,w_n)) \,dx\\
&\quad = \frac{1}{2}\Bigl([(1+s_n)^2-1]\cB_k(u_n)+2(1+s_n)\cB_k(u_n,v_n)+\cB_k(v_n)\Bigr) \\
&\qquad\qquad + \int_{B_R} (F(x,u_n)-F(x,w_n)) \,dx\\
&\quad= -\frac{\|v_n\|^2}{2}+\cB_k\left(u_n,s_n(\frac{s_n}{2}+1)u_n+(1+s_n)v_n\right)+\int_{B_R}(F(x,u_n)-F(x,w_n)) \,dx\\
&\quad\le \cB_k\left(u_n,s_n(\frac{s_n}{2}+1)u_n+(1+s_n)v_n\right)+\int_{B_R}(F(x,u_n)-F(x,w_n)) \,dx
\end{align*}
Since $(u_n)_n$ is a Cerami sequence, $\|v_n\| \le \kappa$ and $|s_n|
\le K+1$ for all $n$, we have 
\begin{align*}
\Bigl|&\cB_k\left(u_n,s_n(\frac{s_n}{2}+1)u_n+(1+s_n)v_n\right) - \irn
f(x,u_n)[s_n(\frac{s_n}{2}+1)u_n+(1+s_n)v_n ]\,dx \Bigr|\\
&= \Bigl|\Phi'(u_n)
\Bigl(s_n(\frac{s_n}{2}+1)u_n+(1+s_n)v_n\Bigr)\Bigr|\\
&\le c_1 \|\Phi'(u_n)\| \|u_n\| +c_2 \|\Phi'(u_n)\| \|v_n\|= o(1)  
\end{align*}
as $n \to \infty$ with constants $c_1,c_2>0$ (depending on $K$). Consequently, 
\begin{align*}
\Phi(w_n)-\Phi(u_n) \le
\int_{B_R}\Bigl(f(x,u_n)[s_n(\frac{s_n}{2}+1)u_n&+(1+s_n)v_n]\\
&+F(x,u_n)-F(x,w_n)\Bigr)\,dx + o(1).    
\end{align*}
Choosing $C=C(K)$ as in Lemma~\ref{lem1extra02}, we conclude that 
$$
\Phi(w_n)-\Phi(u_n) \le |B_R|C + o(1) \qquad \text{as $n \to
  \infty$.}
$$
Hence the assertion follows with $\tilde C= |B_R|C +1.$
\end{proof}

We may now complete the
\begin{proof}[Proof of Proposition~\ref{prop:cerami}]
Let $(u_n)_n$ be a sequence with the assumed properties. We
first show that $(u_n)_n$ is bounded in $X$. Assuming by contradiction
that this is false, we may pass to a subsequence --- still denoted by $(u_n)_n$ --- 
such that $\|u_n\| \to\infty$
as $n\to\infty$.  Setting $w_n:=\frac{u_n}{\|u_n\|}$, we may assume,
passing again to a subsequence, that  
$w_n\weak w$ weakly in $X$ for some $w\in X$, $w_n\to w$ in
$L^q(B_R)$ for all $1\leq q<2^\ast$, and $w_n \to w$ pointwise a.e. on
$B_R$ as $n\to\infty$. Moreover, we have $w_n^+ \weak w^+$ weakly in
$X$ and $w_n^- \to w^-$ strongly in $X$ as $n \to \infty$, since $X^-$
is finite-dimensional. Passing to a further subsequence, we may also assume that 
$\|w_n^+\| \to s \ge 0$ as $n \to \infty$ for some $s \ge \|w^+\|$. Since
$$
o(1)=\frac{\Phi(u_n)}{\|u_n\|^2} \le \frac{1}{2}\Bigl(\|w_n^+\|^2-\|w_n^-\|^2\Bigr),
$$
by \eqref{eq:phi}, we find that 
\begin{equation}
  \label{eq:8}
\|w^-\| \le s.  
\end{equation}
Hence, $1=\|w_n\|^2=\|w_n^+\|^2+\|w_n^-\|^2\to s^2+\|w^-\|^2\le 2s^2$ as $n\to\infty$, 
indicating that $s>0$. Next we suppose by contradiction that 
\begin{equation}
  \label{eq:11}
w \equiv 0 \qquad \text{a.e. on $\Omega$}
\end{equation}
which implies that 
\begin{equation}
  \label{eq:7}
\varphi(t w_n) \to 0 \qquad \text{as $n \to \infty$ for all $t>0$.}  
\end{equation}
We claim that there exist $t>0$ and $v_n \in X^-$, $n \in
\N$ with 
\begin{equation}
  \label{eq:9}
\|v_n\| \le 1 \quad \text{and}\quad \Phi(t w_n +v_n) > c+\tilde
C + 1 \qquad \text{for $n$ sufficiently large,}  
\end{equation}
where $\tilde
C= \tilde C(1)$ is chosen as in Lemma~\ref{sec:proof-theor-refthm:1-1}
corresponding to $\kappa=1$. To prove this, we have to distinguish different cases. We first note that 
$$
\Phi(tw_n)= \frac{t^2}{2} \Bigl(s^2 - \|w^-\|^2\Bigr) +o(1) \quad\text{as }n\to\infty
$$
for every $t>0$ by \eqref{eq:7}. Hence, if $\|w^-\|<s$, we can find $t>0$ such that 
$$
\Phi(t w_n) > c+\tilde C+1 \qquad \text{for $n$ sufficiently large,}
$$
and therefore \eqref{eq:9} follows with $v_n=0$ for every $n$.\\
Next we consider the remaining case $\|w^-\|=s$, and we note that for
every $t>0$ we have $tw_n +w_n^+ \to t w +w^+$ in $L^q(B_R)$
for all $1\leq q<2^\ast$ and also pointwise a.e. on $B_R$. Hence 
\begin{equation}
  \label{eq:7bis}
\varphi(t w_n +w_n^+) \to \varphi(t w +w^+)= \varphi(w^+) \qquad \text{as $n \to \infty$} 
\end{equation}
for all $t>0$ by \eqref{eq:11}. Consequently, 
$$
\Phi(tw_n +w_n^+)= \frac{1}{2}[s^2(t+1)^2 - t^2s^2]-\varphi(w^+)
+o(1)=s^2\left(t+\frac{1}{2}\right) -\varphi(w^+)+o(1),
$$
so that there exists $t>0$ such that 
$$
\Phi(tw_n +w_n^+)>c+\tilde C+1 \qquad \text{for $n$ sufficiently large.}
$$
Again, \eqref{eq:9} follows with $t+1$ in place of $t$ and $v_n=-w_n^-$, 
since $tw_n +w_n^+= (t+1)w_n - w_n^-$ for every $n$. 
Next, fixing $t>0$ and $v_n$, $n \in \N$ such that \eqref{eq:9} holds,
we write $t w_n + v_n= s_n u_n + v_n$ with $s_n=\frac{t}{\|u_n\|}$ for every $n$, 
so that $0<s_n \le 1$ for large $n$
and $\|v_n\| \le \|w_n\| = 1$. By Lemma~\ref{sec:proof-theor-refthm:1-1}, we therefore have  
$$
\Phi(t w_n + v_n)=\Phi(s_n u_n +v_n) \le  \Phi(u_n) +\tilde C +o(1)
$$
as $n \to \infty$, which contradicts \eqref{eq:9}. The contradiction shows that \eqref{eq:11} is false, 
and therefore the set $\Omega_w:=\{x\in \Omega\, :\, w(x)\neq 0\}$ has positive measure. Moreover, 
$$
\text{$|u_n(x)|=\|u_n\| \, |w_n(x)|\to +\infty$ as $n\to\infty$ for
  almost every $x\in \Omega_w$.}
$$
Hence, Fatou's Lemma, the $L^2$-convergence $w_n\to w$ in $B_R$ and $(f_3)$ imply
\begin{align*}
o(1)= \frac{\Phi(u_n)}{\|u_n\|^2}\leq \frac12 - \int_{\Omega_w}\frac{F(x,u_n)}{|u_n|^2}\, |w_n|^2\, dx\to -\infty,
\end{align*}
as $n\to\infty$, a contradiction. The contradiction shows that
$(u_n)_n$ is  bounded in $X$. Therefore, we can find a subsequence --- still denoted by $(u_n)_n$ --- and
some $u\in X$ such that $u_n\weak u$ (weakly)  in $X$, 
$u_n\to u$ in $L^q(B_R)$ for all $1\leq q<2^\ast$ and $u_n(x)\to u(x)$ for a.e. $x\in B_R$.
As a consequence of \eqref{eq:standardestimate}, there holds
$$
\int_{B_R} (f(x,u_n)-f(x,u))(u_n-u)\, dx \to 0\quad \text{as }n\to\infty.
$$
Hence,
\begin{align*}
\cB_0(u_n-u,u_n-u)&= (\Phi'(u_n)-\Phi'(u))(u_n-u) + k^2\int_{B_R}(u_n-u)^2\, dx \\
& \qquad +\int_{B_R}(f(x,u_n)-f(x,u))(u_n-u)\, dx\to 0,
\end{align*}
as $n\to\infty$. It follows from Corollary \ref{sec:some-tools-1} that $u_n\to u$ strongly in $X$.
The proof is finished.
\end{proof}

We may now complete the 
\begin{proof}[Proof of Theorem~\ref{thm:1} (Case $N \ge 2$)] 
(i) The existence of a nontrivial solution follows from a variant of the classical 
linking theorem where the Palais--Smale condition is replaced by the Cerami condition 
\cite[Theorem 2.3]{bartolo-benci-fortunato83}. 
To see this, we proceed as follows. Considering the sequence of
eigenfunctions $(e_j)_{j\in\N}$ of \eqref{eq:10} given by
Lemma~\ref{sec:some-tools-2}, we set $u=e_{j^*} \in X^+$ and put 
$$
Q_\rho:=\{tu+v\, :\, v\in X^-, \ \|v\|\leq \rho, \ 0\leq t\leq \rho\}
\qquad \text{for $\rho>0$.}
$$
Note that the sets $Q_\rho$ are contained in the finite dimensional
subspace $W= X^- \oplus \R u \subset X$. Since every function in $W$ is
analytic in $B_R$, $\cZ:=W$ satisfies condition \eqref{eq:19}. Using  
Lemma \ref{lemma-new1} (ii) and the fact that $\Phi$ is nonpositive on
$X^-$ by (\ref{eq:phi}), we thus find that 
$$
 \sup_{\partial Q_\rho}\Phi = 0 \qquad \text{for $\rho>0$ sufficiently
   large.}
$$
According to Lemma~\ref{lemma-new1} (i), we may further choose $\alpha>0$
sufficiently small such that the sets $\Sigma_\alpha$ and $\partial Q_\rho$
link and $\inf\limits_{\Sigma_\alpha}\Phi>0$ (see e.g. \cite{struwe}
or \cite[Section 2]{bartolo-benci-fortunato83} for the notion of linking of sets). Finally,
we have $\sup\limits_{Q_\rho}\Phi<+\infty$ by the compactness of $Q_\rho$. 
Taking Proposition~\ref{prop:cerami} into account, we can apply the linking theorem and 
obtain that $\Phi$ has a nontrivial critical point $\hat u\in X$ such that 
$$
0 < \inf\limits_{\Sigma_\alpha}\Phi \le \Phi(\hat u) \le
\sup\limits_{Q_\rho}\Phi.
$$
(ii) Let us now assume that $f(x,-t)=-f(x,t)$ for all
$(x,t)\in\R^N\times\R$. In this case, a variant of the Fountain
Theorem (see \cite{bartsch93,bartsch-willem93} and \cite[Theorem
3.6]{willem}) yields the existence of a sequence of
pairs $\{\pm u_n\}$, $n \in \N$ of critical points of $\Phi$ such that 
\begin{equation}
  \label{eq:29}
\Phi(u_n) \to \infty \qquad \text{as $n \to \infty$.}  
\end{equation}
More precisely, we use a version of the Fountain Theorem where the
Cerami condition is used in place of the Palais--Smale condition. To
see that such a variant exists, it suffices to
note the validity of a deformation lemma giving rise to Cerami
sequences instead of Palais--Smale sequences. Such a deformation lemma
has already been established in \cite[Theorem 1.3]{bartolo-benci-fortunato83}.
In order to check the other assumptions of the Fountain Theorem, we remark that
$X=\overline{\bigoplus \limits_{j\in\N} \R e_j}$
where $(e_j)_{j\in\N}$ is given by Lemma~\ref{sec:some-tools-2} (i). We set 
$$
X_j=\R e_j ,\quad  Y_j=\bigoplus \limits_{\ell=1}^{j}X_\ell\quad
\text{and}\quad  Z_j=\overline{\bigoplus \limits_{\ell=j}^\infty
  X_\ell}
$$
for $j\in\N$. Since every function in $Y_j$ is analytic, we see from
Lemma \ref{lemma-new1} (ii), applied to $\cZ=Y_j$, that for every $j \in\N$ there exists $\rho_j>0$
such that $\Phi(u)\leq 0$ for $u\in Y_j$ with $\|u\|\geq\rho_j$. It
only remains to check that for some sequence $(r_j)_j\subset(0,\infty)$
\begin{equation}
  \label{eq:23}
\text{$\inf\{\Phi(u)\, :\, u\in Z_j,\ \|u\|=r_j\}\to\infty$ as $j \to \infty$.}
\end{equation}
This will be proved by similar arguments as in
\cite[Theorem 3.7]{willem}. Indeed, if $j \ge j^*$, then $Z_j\subset
X^+$ and therefore, by \eqref{eq:standardestimate},
\begin{align}
\Phi(u)=\frac12\|u\|^2-\int_{B_R}F(x,u)\, dx&\geq
\frac12\|u\|^2-\frac{\eps}{2}\|u\|_{L^2(B_R)}^2-\frac{C_\eps}{p}\|u\|_{L^p(B_R)}^p
\nonumber \\
&\geq \frac{1}{4}\|u\|^2 -\frac{C_\eps}{p}\beta_j^p\|u\|^p \qquad
\text{for $u \in Z_j$}, \label{eq:A3_Fountain}
\end{align}
where 
$$
\eps=\frac12\inf\limits_{u\in
  X\backslash\{0\}}\frac{\|u\|^2}{\|u\|_{L^2(B_R)}^2}>0\quad
\text{and}\quad \beta_j:=\sup \{\|u\|_{L^p(B_R)}\::\: u\in
  Z_j,\:\|u\|=1\}.
$$ 
Since $(\beta_j)_j\subset[0,\infty)$ is a 
decreasing sequence, $\beta:=\lim\limits_{j\to\infty}\beta_j$ exists.
Moreover, for each $j$ we can find some $u_j\in Z_j$ such that
$\|u_j\|=1$ and $\beta_j\ge 
\|u_j\|_{L^p(B_R)}> \frac{\beta_j}{2}$. From the definition of $Z_j$ we obtain that 
$u_j\weak 0$ (weakly) in $X$ and the compact Sobolev embedding $X\embed L^p(B_R)$ then gives $u_j\to 0$
in $L^p(B_R)$. Thus, $\beta_j\to 0$ as $j\to\infty$. Choosing $r_j=(2C_\eps\beta_j^p)^{\frac{1}{2-p}}$, 
we obtain from \eqref{eq:A3_Fountain} that
$$
\Phi(u)\geq r_j^2(\frac14-\frac{1}{2p})\qquad \text{for all $u\in Z_j$
  with $\|u\|=r_j$.}
$$
Since $r_j\to\infty$ as $j\to\infty$,
the assertion follows. Moreover, since $\Phi$ is bounded on bounded
subsets of $X=H^1(B_R)$, \eqref{eq:29} implies that
$\|u_n\|_{H^1(B_R)} \to \infty$ as $n \to \infty$, as claimed in (\ref{eq:28}).\\
To conclude the proof of Theorem \ref{thm:1}, we remark that if $u$ is a critical point of $\Phi$, then the restriction 
of $u$ to $S_R$ belongs to $H^{\frac12}(S_R)$ and there exists a unique weak solution $w$ of \eqref{eq:2}. 
Therefore, extending $u$ on $\R^N$ by setting $u \equiv \text{Re}(w)$
on $E_R$, we see that $u \in \cR_R$ is a solution of \eqref{eq:33}.
\end{proof}

We close this section with an observation on a rigid minimax
characterization of the ground state energy
level (i.e., the least energy of a nontrivial critical point of
$\Phi$) in the case where $(f_4)$ is replaced by a stronger condition.

\begin{theorem}
\label{sec:proof-theor-refthm:1-2}
Suppose that the nonlinearity $f$ satisfies $(f_1)-(f_3)$ and the
following stronger version of $(f_4)$: For every $x \in \R^N$,
\begin{equation}
  \label{eq:13-1}
\text{$u \mapsto \frac{f(x,u)}{|u|}$ is nondecreasing on
  $\R\setminus\{0\}$.}    
\end{equation}
Then the ground state energy level 
$$
c:= \inf \{\Phi(u)\,:\, \text{$u$ nontrivial critical point of
  $\Phi$}\}
$$
is positive and equivalently given by 
$$
c= \inf_{u \in X^+ \setminus \{0\}}\: \sup_{t \ge 0, v \in X^-}\Phi(tu+v)
$$
Moreover, $c$ is attained within the set of nontrivial critical points
of $\Phi$, i.e., within the set of nontrivial solutions of \eqref{eq:33_i}.   
\end{theorem}

\begin{proof}
Let $u$ be a nontrivial critical point of $\Phi$. Then $u \not \in
X^-$, since otherwise
$$
\int_{B_R}f(x,u)u\,dx = \cB_k(u) = -\|u^-\|^2 <0,
$$
contrary to \eqref{eq:13-1} and $(f_2)$. Next we claim that 
\begin{equation}
\label{eq:15}
\Phi(u) \ge \Phi(w) \qquad \text{for every $w \in X^-\! + \R^+ u\,=\, X^-\! + \R^+ u^+$.}
\end{equation}
Indeed, similar to the proof of
Lemma~\ref{sec:proof-theor-refthm:1-1} we have, for $w=(1+s)u+v$
with $s \ge -1$ and $v \in X^-$,  
\begin{align*}
\Phi (w)-\Phi(u) &\le \cB_k(u,s(\frac{s}{2}+1)u+(1+s)v)+
\int_{B_R}(F(x,u)-F(x,w)) \,dx\\
&= \int_{B_R}\Bigl(f(x,u)[s(\frac{s}{2}+1)u+(1+s)v]+F(x,u)-F(x,w)\Bigr) \,dx. 
\end{align*}
Moreover, the last integral is nonpositive since, as a consequence of
\eqref{eq:13-1},  the inequality \eqref{eq:14} holds for
$f$ in place of $f_1$. Hence \eqref{eq:15} is true. Putting 
$$
c^*= \inf_{u \in X^+ \setminus \{0\}}\: \sup_{t \ge 0, v \in
  X^-}\Phi(tu+v),
$$
we thus infer that $c \ge c^*$, whereas Lemma
\ref{lemma-new1} (i) implies that $c^* \ge \inf
\limits_{\Sigma_\alpha}\Phi>0$ for some $\alpha>0$, where
$\Sigma_\alpha:= \{u \in X^+\,:\,\|u\|=\alpha\}$.\\ 
Furthermore, $c$ is attained among nontrivial
critical points of $\Phi$. Indeed, if
$(u_n)_n$ is a sequence of critical points of $\Phi$ with $\Phi(u_n)
\to c$, then by Proposition~\ref{prop:cerami} we have $u_n \to u_0$ in
$X$ for a subsequence, where $u_0$ is a critical point of $\Phi$ with
$\Phi(u_0)=c>0$, and therefore $u_0 \not=0$.\\
It remains to show that $c \le c^*$. For this, let $u \in X^+
\setminus \{0\}$ be such
that 
\begin{equation}
  \label{eq:22}
\sup_{t \ge 0, v \in X^-}\Phi(tu+v)<\infty.
\end{equation}
Let $\cZ:= \{tu +v\::\: v \in X^-,\: t \ge 0\}$. We claim that condition
\eqref{eq:19} holds for this set $\cZ$. Suppose by contradiction
that this is false, i.e., there exists $w \in \cZ
\setminus \{0\}$ with $\|w^+\| \ge \|w^-\|$ and such that $w \equiv 0$
a.e. on $\Omega$. Since $\|w^+\|>0$, we then find  
\begin{align*}
\Phi(tw+w^+)&= \frac{1}{2}
\Bigl((t+1)^2\|w^+\|^2-t^2\|w^-\|^2\Bigr)- \varphi(w^+)\\
&\ge(t+\frac{1}{2}) \|w^+\|^2- \varphi(w^+) \to \infty,
\end{align*}
as $t \to \infty$. This contradicts \eqref{eq:22}, since $tw + w^+=
(t+1)w^++tw^- \in \cZ$ for every $t>0$. The contradiction
shows that condition
\eqref{eq:19} holds for $\cZ$. Since $\cZ$ is a closed cone contained
in the finite dimensional space $X^- \oplus \R u$, Lemma~\ref{lemma-new1} (ii)
implies that there exists $\rho>0$ with 
\begin{equation}
  \label{eq:20}
 \sup_{\partial Q_\rho}\Phi = 0 \qquad \quad \text{where }\quad Q_{\rho}:=\{tu+v\, :\,
 v\in X^-, \ \|v\| \leq \rho, \ 0\leq t\leq \rho\}.
\end{equation}
Applying the linking theorem \cite[Theorem 2.3]{bartolo-benci-fortunato83} exactly as in
the proof of Theorem~\ref{thm:1} with the present choice of $Q_\rho$, we obtain a 
nontrivial critical point $\hat u$ of $\Phi$ with 
$$
\Phi(\hat u) \le \max_{Q_\rho} \Phi \le \sup_{t \ge
  0,v \in X^-}\Phi(tu+v).
$$
We thus conclude that $c \le c^*$, so equality holds
and the proof is finished.
\end{proof}

\section{The one-dimensional case}\label{sec-remarks}
In this section we sketch the proof of Theorem~\ref{thm:1} in the case
$N=1$ which has to be treated slightly differently.
Note that for $R>0$ we have $B_R=(-R,R)$, $S_R:= \{-R,R\}$ and
$E_R=(-\infty,-R)\cup(R,+\infty)$, and consider the space $H_R^\C$
as defined in \eqref{eq:1}. For a given function $u:S_R \to \C$, it is
easy to see that the unique solution $w \in H_R^\C$ of the problem  
\begin{equation}
\left\{  
\begin{aligned}
w'' +k^2 w&=0 &&\qquad \text{in $E_R$}, \\ 
w&=u  && \qquad \text{on $S_R$}     
  \end{aligned}
\right.
\end{equation}
is then given by 
$$
w(x)=\left\{
  \begin{aligned}
&u(-R) e^{-ik(x+R)}&&\qquad  \text{for }x\leq -R,\\
&u(R) e^{ik(x-R)} &&\qquad \text{for }x\geq R.
  \end{aligned}
\right.
$$
Hence the capacity operator $T_R$ is simply given by $[T_R u](R)= ik u(R)$ and \\ 
$[T_R u](-R)= ik u(-R)$ for any function $u:S_R \to \C$, and thus its real
part $K_R$, as defined in \eqref{eq:27}, is zero. In particular, the second inequality in
\eqref{eq:5} is not true if $N=1$, and also Corollary~\ref{sec:some-tools-1}
does not hold in this case. On the other hand, the quadratic form 
$\cB_k$ defined in \eqref{eq:25} is simply given by 
$$
\cB_k(u,v)= \int_{-R}^R (u' v' -k^2 u v)\,dx, \qquad u,v \in H^1(-R,R),
$$
and \eqref{eq:10} reduces to the homogeneous Neumann
eigenvalue problem     
\begin{equation}
  \label{eq:26}
-u'' = \lambda u \quad \text{in $(-R,R)$,}\qquad u'(\pm R)=0.
\end{equation}
As a consequence, the properties listed in
Lemma~\ref{sec:some-tools-2} are well known and easy to check in the
case $N=1$. Furthermore, the inner problem \eqref{eq:33_i} reduces to 
\begin{equation}\label{eq:33_ii}
\left\{
  \begin{aligned}
    -u'' -k^2 u&=f(x,u) &&\qquad \text{in $(-R,R)$}, \\
    u'(\pm R)&=0,
  \end{aligned}
\right.
\end{equation}
and solutions of \eqref{eq:33_ii} correspond to
critical points of the functional 
$$
\Phi: H^1(-R,R) \to \R, \qquad \Phi(u)= \frac12 \int_{-R}^R \Bigl((u')^2 -k^2 u^2\Bigr)\, dx 
- \int_{-R}^R F(x,u)\, dx.
$$
Since the structure of $\Phi$ is the same as in the multidimensional case, the proof 
of Theorem~\ref{thm:1} can be finished exactly
as detailed in Section~\ref{sec-thm-1}. In particular, the
one-dimensional analogues of Lemma~\ref{lemma-new1} and
Proposition~\ref{prop:cerami} are valid.

In addition, we remark that, by exactly the same proof,
Theorem~\ref{sec:proof-theor-refthm:1-2} is also true in the
one-dimensional case. 

\section{Existence of radial solutions}\label{sec-radial}
For $N\geq 2$, we look for radially symmetric solutions of
\begin{equation}\label{eqn:3.1}
	-\Delta u -k^2 u = f(|x|,u), \qquad x\in\R^N,
\end{equation}
where $k>0$, and $f$ satisfies the following assumptions:
\begin{itemize}
	\item[$(g_1)$] $f\in\cC([0,\infty)\times\R)$ with $f(\cdot,0)=0$, and $f$ is 
	locally Lipschitz continuous with respect to $u\in\R$.
	\item[$(g_2)$] $0\leq F(r,u)\leq \frac{1}{2}f(r,u)u$ for all $(r,u)\in[0,\infty)\times\R$, 
	where $F(r,u)=\int\limits_0^u f(r,s)\, ds$.
	\item[$(g_3)$] For every $u\in\R$, $F(\cdot,u)\in\cC^1(0,\infty)$ with $\partial_r F(r,u)\leq 0$ for all $r>0$.
\end{itemize}
Setting $u(x)=u(r)$ with $r=|x|$, we can rewrite \eqref{eqn:3.1} as
\begin{equation}\label{eqn:3.2}
	-u'' -\frac{N-1}{r}u' -k^2 u = f(r,u), \quad r>0
\end{equation}
and we shall look for solutions of class $\cC^2$ of this equation which satisfy $u'(0)=0$. 

The following local existence and uniqueness result (in a singular setting) is well known, see
e.g. \cite[Lemma 3.1]{weth05} for a proof. 

\begin{lemma}\label{lem:exist_eps}
For every $\alpha\in\R$, there exists $\eps>0$ such that the initial value problem associated with \eqref{eqn:3.2}, 
$u(0)=\alpha$ and $u'(0)=0$ has a
unique solution $u\in\cC^2([0,\eps),\R)$.
\end{lemma}

Using properties $(g_2)$ and $(g_3)$, we now continue the local
solution given by Lemma~\ref{lem:exist_eps} to a global solution on
$\R^+$ and analyze its asymptotic behavior.

\begin{theorem}\label{thm:radial}
Let $f$ satisfy $(g_1)$ to $(g_3)$. Then for every $\alpha\in\R$, there exists a unique radially 
symmetric solution of \eqref{eqn:3.1}, $u=u(|x|)$,
$u\in\cC^2([0,\infty),\R)$ satisfying $u(0)=\alpha$ and
$u'(0)=0$. Moreover, 
\begin{equation}
  \label{eq:24}
\sup_{r \ge 1} \left\{r^{N-1}\bigl(u'(r)^2+u(r)^2\bigr)\right\} <\infty,
 \end{equation}
so that --- identifying $u(x)$ with $u(|x|)$ for $x \in \R^N$ --- we have
$\frac{u}{\sqrt{1+|x|^2}}$, $\frac{|\nabla
  u|}{\sqrt{1+|x|^2}}\in L^2(\R^N)$. If, in addition,
\begin{itemize}
\item[$(g_4)$] $f(r,u)=o(|u|)$ as $u\to 0$, uniformly in $r>0$, 
\end{itemize}
then $u$ satisfies the asymptotic condition \eqref{eqn:radiation2}.
\end{theorem}

\begin{proof}
Let $\alpha>0$. By Lemma \ref{lem:exist_eps}, there exists $\eps>0$
and a unique solution $u\in\cC^2([0,2\eps),\R)$ satisfying
\eqref{eqn:3.2} and $u(0)=\alpha$, $u'(0)=0$. Let $[0,r_0)$, $r_0 \in
[2\eps,\infty]$ denote the maximal existence interval of this
solution. We claim that 
\begin{equation}
  \label{eq:24-1}
\sup_{r \in [\eps,r_0)} r^{N-1}\bigl( u'(r)^2+u(r)^2 \bigr) <\infty,
 \end{equation}
To see this, we use the change of variables 
$$
v(r) = r^{\frac{N-1}{2}} u(r),\quad r \in (0,r_0),
$$
so that $v$ solves 
\begin{equation}\label{eqn:3.5}
-v'' -\{k^2-\frac{(N-1)(N-3)}{4r^2}\}v =
r^{\frac{N-1}{2}}f(r,r^{-\frac{N-1}{2}}v),\quad r \in (0,r_0).
\end{equation}
We consider the $\cC^1$-function $r \mapsto \rho(r):= v'(r)^2 + k^2
v(r)^2$ on $(0,r_0)$ which satisfies
$$
\frac12 \rho'(r) = \frac{(N-1)(N-3)}{4r^2}v(r)v'(r) - r^{\frac{N-1}{2}}f(r,r^{-\frac{N-1}{2}}v(r))v'(r).
$$
For $r \in [\eps,r_0)$, we thus obtain 
\begin{align*}
\rho(r)&= \rho(\eps) + 2\!\!\int_\eps^r \frac{(N-1)(N-3)}{4s^2}v(s)v'(s)\, ds 
 - 2\!\!\int_\eps^r s^{\frac{N-1}{2}}f(s,s^{-\frac{N-1}{2}}v(s))v'(s)\, ds \\
        & = \rho(\eps) + 2\int_\eps^r \frac{(N-1)(N-3)}{4s^2}v(s)v'(s)\, ds 
 - 2 s^{N-1} F(s,s^{-\frac{N-1}{2}}v(s))\Bigr|_\eps^r \\
	& \quad - (N-1)\int_\eps^r s^{N-2}\{ f(s,s^{-\frac{N-1}{2}}v(s))s^{-\frac{N-1}{2}}v(s) 
 - 2F(s,s^{-\frac{N-1}{2}}v(s))\}\, ds \\
	& \qquad + 2\int_\eps^r s^{N-1}\partial_r F(s,s^{-\frac{N-1}{2}}v(s))\, ds \\
	&\leq \rho(\eps) + 2\eps^{N-1}F(\eps,\eps^{\frac{1-N}{2}}v(\eps)) 
 + 2\int_\eps^r \frac{(N-1)(N-3)}{4s^2}v(s)v'(s)\, ds \\
	& \leq \rho(\eps) + 2\eps^{N-1}F(\eps,\eps^{\frac{1-N}{2}}v(\eps)) 
 + \int_\eps^r \frac{(N-1)|N-3|}{4ks^2} \rho(s)\, ds,
\end{align*}
using $(g_2)$ and $(g_3)$.
Hence, Gronwall's inequality (see \cite[Theorem III.1.1]{hartman}) gives
$$
\rho(r) \leq [\rho(\eps) + 2\eps^{N-1}F(\eps,\eps^{\frac{1-N}{2}}v(\eps))] 
e^{\frac{(N-1)|N-3|}{4k\eps^2}} \quad \text{for }r \in [\eps,r_0).
$$
From this we derive \eqref{eq:24-1}, since 
$$
r^{N-1}\bigl(u'(r)^2+u(r)^2\bigr) \le C_\eps \rho(r) \quad
  \text{for $r \in [\eps,r_0)$ with some constant $C_\eps>0$.}
$$
As a consequence of \eqref{eq:24-1}, we find that $r_0= \infty$ and
that (\ref{eq:24}) holds.\\
We finally assume that $(g_4)$ holds, and we check that the radiation condition in
\eqref{eqn:radiation2} is fulfilled. For this we let $r\geq 1$ and write
$$
r^{\frac{N-1}{2}}\left|u''(r)+k^2u(r)\right|=\left|\frac{(N-1)}{r}\!\left(\frac{(N-1)}{2r}v(r) - v'(r)\right) 
- \frac{f(r,r^{-\frac{(N-1)}{2}}v(r))}{r^{-\frac{(N-1)}{2}}v(r)}\, v(r)\right|
$$
$$
\leq \frac{(N-1)^2}{2r^2}\|v\|_\infty+ \frac{(N-1)}{r}\|v'\|_\infty 
+ \left|\frac{f(r,r^{-\frac{(N-1)}{2}}v(r))}{r^{-\frac{(N-1)}{2}}v(r)}\right|\, \|v\|_\infty.
$$
Since $\lim\limits_{r\to\infty}r^{-\frac{(N-1)}{2}}v(r)=0$, the assumption $(g_4)$ gives
$
\lim\limits_{r\to\infty}r^{\frac{N-1}{2}}\left|u''(r)+k^2u(r)\right|=0,
$
i.e., condition \eqref{eqn:radiation2} holds.
\end{proof}

We point out that if $f$ satisfies $(g_1)$ -- $(g_4)$, every solution of \eqref{eqn:3.1} given by 
Theorem \ref{thm:radial} is oscillatory. This follows from the fact that for every such solution $u$, 
$0$ lies in the essential spectrum of any self-adjoint realization of the Sturm--Liouville differential expression
$$
\tau w=-w''-\Bigl(k^2-\frac{(N-1)(N-3)}{r^2}+\frac{f(r,u(r))}{u(r)}\Bigr)w
$$
associated to \eqref{eqn:3.5} (see \cite[Corollary XIII.7.14 and Theorem XIII.7.40]{dunford-schwartz}).

\section{Properties of the capacity operator and its real part}
\label{sec-prop-capac-oper}
 In this section we derive key properties of the operator $T_R$
introduced in Section~\ref{sec-some-tools} and
its real part $K_R$ in the case $N \ge 2$ (see
Section~\ref{sec-remarks} for the case $N=1$). In particular, we will give the proofs of
Lemma~\ref{sec:capacity-operator} and Lemma~\ref{sec:some-tools-2} (iv). 
Part of the material presented in this section should be well known
to experts and appears in the standard literature in the special case
$N=3$ (see e.g. \cite{nedelec}). However we could not find the
corresponding formulas for the general $N$-dimensional case. Moreover,
as already remarked in the introduction, it seems that the real part of the capacity operator has not
been studied in the degree of detail which we need for our
purposes. 

Let $N\geq 2$, and consider the exterior Dirichlet problem for the Helmholtz equation
\begin{equation}
  \label{eq:app1}
\left\{
  \begin{aligned}
\Delta w + k^2 w &= 0 &&\qquad \text{in $E_R:=\R^N\backslash \overline{B_R}$,}\\
w&=u &&\qquad \text{on $S_R$},  
  \end{aligned}
\right.
\end{equation}
where $k>0$, $R>0$ and $u\in H^{\frac12}(S_R,\C)$ are given. Recall 
the space $H_R^{\C}$ defined in \eqref{eq:1} which enforces a weak
form of Sommerfeld's radiation condition. We wish
to construct the unique solution $w \in H_R^{\C}$ explicitly in terms
of a Fourier representation of the boundary datum $u$. For this we let $\Delta_S$
denote the Laplace--Beltrami operator on the unit sphere
$S_1\subset\R^N$, so that 
\begin{align}\label{eq:laplace-beltrami}
\Delta w &= \frac{1}{r^{N-1}}\frac{\partial}{\partial r}\left( r^{N-1}\frac{\partial w}{\partial r}\right) 
+ \frac{1}{r^2}\, \Delta_S w.
\end{align}
We recall (see e.g. \cite[Corollary 2.3]{stein-weiss}) that the linear span of the spherical harmonics 
(i.e. restrictions to $S_1$ of harmonic polynomials on $\R^N$ with complex coefficients) is dense in $L^2(S_1,\C)$. 
More precisely, denoting by $\cH^N_\ell$ the space of spherical harmonics of degree $\ell\in\N_0$, 
then $d_\ell^N:=\dim \cH^N_\ell = \frac{(N+2\ell-2)}{(N+\ell-2)}
\left(\substack{ N+\ell-2 \\
\ell}\right)$ 
if $\ell\geq 1$, $d_0^N:=\dim \cH^N_0=1$, and $\spann \bigcup\limits_{\ell=0}^\infty \cH^N_\ell$ is a dense subspace 
of $L^2(S_1,\C)$. Furthermore, according to \eqref{eq:laplace-beltrami}, every nonzero $\cY_\ell\in \cH^N_\ell$ is an 
eigenfunction of $\Delta_S$:
\begin{align}\label{eq:lb-eigen}
\Delta_S \cY_\ell + \ell(N+\ell-2) \cY_\ell =0 \quad\text{in }S_1
\end{align}
see \cite[Theorem 3.2.11]{groemer}. Moreover, starting from the orthogonal bases $\{1\}$ of $\cH^2_0$ and 
$\{e^{i\ell\theta}, e^{-i\ell\theta}\}$ of $\cH^2_\ell$, $\ell\geq 1$, 
($\theta=\theta(x_1,x_2)=\text{sgn}(x_2)\arccos x_1$) 
we can inductively construct, according to \cite[Lemma 3.5.3]{groemer}, orthogonal bases of $\cH^N_\ell$
$\{\cY^\ell_1,\ldots, \cY^\ell_{d_\ell^N}\}$, $\ell\in\N_0$ for all 
$N\geq 2$, with the property that for each element $\cY$ in such a basis, the element 
$\overline{\cY}$ also belongs to this basis.

\begin{proposition}
\label{sec:prop-capac-oper-1}
For $u \in H^{\frac12}(S_R)$ given by the expansion $u(R\xi)=\sum
\limits_{\ell=0}^\infty \sum \limits_{m=1}^{d^N_\ell} u^\ell_m
\cY^\ell_m(\xi)$, 

\noindent $\xi \in S_1$, the problem \eqref{eq:app1} has a
unique solution $w \in H_R^{\C}$ given by 
\begin{equation}
  \label{eq:4}
 w(x)= \left(\frac{r}{R}\right)^{-\frac{N-2}{2}} \sum_{\ell=0}^\infty  
 \frac{H^{(1)}_{\mu_\ell}(kr)}{H^{(1)}_{\mu_\ell}(kR)} \sum_{m=1}^{d^N_\ell} u^\ell_m \cY^\ell_m(\xi) 
 \quad \text{for $r \ge R$, $\xi \in S_1$ and $x=r\xi$,}
 \end{equation}
where, here and in the following, $\mu_\ell=\ell+\frac{(N-2)}{2}$ for $\ell \in \N_0$ and
$H^{(1)}_\mu$, $H^{(2)}_\mu$ are the two Hankel functions (or Bessel
functions of the third kind) of order $\mu$ (see e.g. \cite[\S
3.6]{watson}).
\end{proposition}

\begin{proof}
We start by rescaling the space variable $x$ to $\frac{x}{R}$. Setting $w(x)=v(\frac{x}{R})$, 
the problem \eqref{eq:app1} becomes
\begin{equation}
  \label{eq:app2}
\left\{
  \begin{aligned}
\Delta v + (kR)^2 v &= 0 &&\qquad \text{in $E_1$,}\\
v&=u_R &&\qquad \text{on $S_1$},  
  \end{aligned}
\right.
\end{equation}
where $u_R(\xi)=u(R\xi)$, $\xi\in S_1$. We first remark that $w\in H_R^{\C}$ if and only if $v$ belongs to
the space
\begin{align}
H_{1,kR}^{\C}:= \Bigl\{u \in H^1_{loc}(E_1,\C)\::\: &\frac{u}{(1+r^2)^{1/2}} \in
  L^2(E_1,\C),\: \frac{\nabla u}{(1+r^2)^{1/2}} \in
  L^2(E_1,\C^N), \nonumber\\
&\frac{\partial u}{\partial r}  -i kR u \in
 L^2(E_1,\C)  \label{eq:111}
\Bigr\}.
\end{align}
We will therefore seek in this space, solutions of the form $v^\ell_m(x)=f_\ell(r)\cY^\ell_m(\xi)$ for
some $\ell \in \N_0$, $1 \le m \le d_l^N$. Using \eqref{eq:laplace-beltrami} and~\eqref{eq:lb-eigen}, we may
rewrite \eqref{eq:app2} for functions of this form to obtain
\begin{align*}
   0 &=\{f_\ell''(r) +\frac{(N-1)}{r}f'_\ell(r) +\left((kR)^2 -\frac{\ell(N+\ell-2)}{r^2}\right)f_\ell(r)\} \cY^\ell_m(\xi).
\end{align*}
Setting $g_\ell(s)=(\frac{s}{kR})^{\frac{N-2}{2}}f_\ell(\frac{s}{kR})$, we find
$$
g_\ell''(s) + \frac{1}{s}g_\ell'(s) + \left\{1-\frac{(\ell+\frac{N-2}{2})^2}{s^2}\right\}g_\ell(s)=0,\quad s>0,
$$
which is Bessel's equation with parameter $\mu=\ell+\frac{N-2}{2}$. Its general solution is given by
$$
g(s)= A H^{(1)}_\mu(s) + B H^{(2)}_\mu(s),
$$
where $A, B\in\C$, and $H^{(1)}_\mu$, $H^{(2)}_\mu$ are the two Hankel
functions of order $\mu$.
Now, we observe that the Sommerfeld radiation condition, given as the
third condition in the definition of the space $H_{1,kR}^{\C}$ in \eqref{eq:111}, can be rewritten in the form
\begin{align*}
 &\int_{kR}^\infty s^{-1} \bigl| 2A s[ (H^{(1)}_\mu)'(s) -i H^{(1)}_\mu(s)] -A(N-2) H^{(1)}_\mu(s) \\
 &\qquad + 2B s[ (H^{(2)}_\mu)'(s) -i H^{(2)}_\mu(s)] -B(N-2)H^{(2)}_\mu(s)\bigr|^2\, ds<\infty.
\end{align*}
Using the recurrence formula $\displaystyle (H^{(p)}_\mu)'(s) = \frac{\mu}{s} H^{(p)}_\mu(s) - H^{(p)}_{\mu+1}(s)$, 
$p=1,2$, we obtain
\begin{align}
 &\int_{kR}^\infty s^{-1} \bigl| 2A s[ H^{(1)}_{\mu+1}(s) +i H^{(1)}_\mu(s)] +A(N-2-2\mu) H^{(1)}_\mu(s) \nonumber \\
 &\qquad + 2B s[ H^{(2)}_{\mu+1}(s) +i H^{(2)}_\mu(s)] +B(N-2-2\mu)H^{(2)}_\mu(s)\bigr|^2\, ds<\infty. \label{eq:somm2}
\end{align}
According to \cite[Formulas 7.2 (1) and (2)]{watson}, the asymptotic behavior of $H^{(p)}_\mu(s)$, $p=1,2$, is given by
$$
H^{(1)}_\mu(s)=\sqrt{\frac{2}{\pi s}}\, e^{i(s-\frac{2\mu+1}{4}\pi)}[1+O(s^{-1})],\;
 H^{(2)}_\mu(s)=\sqrt{\frac{2}{\pi s}}e^{-i(s-\frac{2\mu+1}{4}\pi)}[1+O(s^{-1})].
$$
As a consequence, 
$$
|H^{(p)}_\mu(s)|=s^{-\frac12}\{\sqrt{\frac{2}{\pi}}+O(s^{-1})\} \qquad
\text{for $p=1,2$,}
$$
$$ 
s|H^{(1)}_{\mu+1}(s)+i H^{(1)}_\mu(s)|= O(s^{-\frac12})\;\text{ and }\; s|H^{(2)}_{\mu+1}(s)+i H^{(2)}_\mu(s)|
= s^{\frac12}[2\sqrt{\frac{2}{\pi}}+O(s^{-1})].
$$
Thus, \eqref{eq:somm2} can only be satisfied, if $B=0$. We conclude that a function in $H_{1,kR}^{\C}$ of the
form $v^\ell_m(x)=f_\ell(r) Y^\ell_m(\xi)$ is a solution of the
differential equation in \eqref{eq:app2} if and only if it can be written as
$$
v^\ell_m(x)= A r^{-\frac{(N-2)}{2}}\, H^{(1)}_{\mu_\ell}(kRr)\, \cY^\ell_m(\xi)
$$
for some $A\in\C$ with $\mu_\ell= \ell +\frac{N-2}{2}$. Since
$u_R(\xi)=\sum \limits_{\ell=0}^\infty
\sum \limits_{m=1}^{d^N_\ell} u^\ell_m \cY^\ell_m(\xi)$, a solution of the
boundary value problem \eqref{eq:app2} is thus given by
$$
 v(x)=r^{-\frac{(N-2)}{2}} \sum_{\ell=0}^\infty \frac{H^{(1)}_{\mu_\ell}(kRr)}{H^{(1)}_{\mu_\ell}(kR)} 
 \sum_{m=1}^{d^N_\ell} u^\ell_m \cY^\ell_m(\xi)
 \quad \text{for $r \ge 1$, $\xi \in S_1$ and $x=r\xi$.}
$$
Rescaling back, we thus obtain the formula \eqref{eq:4} for the
(unique) solution $w \in H_R^\C$ of \eqref{eq:app1}.
\end{proof}

The formula~\eqref{eq:4} gives rise to the following expression for the capacity
operator $T_R: H^{\frac12}(S_R,\C) \to H^{-\frac12}(S_R,\C)$ (see
\cite{nedelec} for the case $N=3$):
\begin{equation}
 [T_Ru](R\xi) = \frac{1}{R}\sum_{\ell=0}^\infty z_\ell(kR)
 \sum_{m=1}^{d^N_\ell} u^\ell_m \cY^\ell_m(\xi),\qquad \xi \in S_1
\end{equation}
for $u \in H^{\frac12}(S_R,\C)$ given by $u(R\xi)=\sum \limits_{\ell=0}^\infty
\sum \limits_{m=1}^{d^N_\ell} u^\ell_m \cY^\ell_m(\xi)$, $\xi \in S_1$, where
$$
 z_\ell(r) = \frac{r\frac{d}{dr}H^{(1)}_{\mu_\ell}(r)}{H^{(1)}_{\mu_\ell}(r)}-\frac{N-2}{2},\quad r>0.
$$

We need some estimates for the coefficients $z_\ell(r)$. For this we
recall that 
$$
H^{(1)}_\mu=J_\mu+i Y_\mu,\quad \text{and}\quad H^{(2)}_\mu= J_\mu - i Y_\mu=
\overline{H^{(1)}_\mu} \qquad \text{for $\mu \in \R$,}  
$$
where $J_\mu$ and $Y_\mu$ denote the Bessel functions of the first and
second kind. Setting  
$G_\mu(r)=\frac{r\frac{d}{dr}H^{(1)}_\mu(r)}{H^{(1)}_\mu(r)}$, for $r>0$,
$\mu\in\R$, we see that 
\begin{align}
 \text{Re }G_\mu(r) &=\frac{r\,\text{Re}(\overline{H^{(1)}_\mu(r)}\frac{d}{dr}H^{(1)}_{\mu}(r))}{|H^{(1)}_\mu(r)|^2}
= \frac{r}{2} \frac{\frac{d}{dr}|H^{(1)}_\mu(r)|^2}{|H^{(1)}_\mu(r)|^2}
 =\frac{r}{2}\frac{\frac{d}{dr}(J_\mu^2(r)+Y_\mu^2(r))}{J_\mu^2(r)+Y_\mu^2(r)} \label{eqn:ReG}\\
 \text{Im }G_\mu(r) &=\frac{r\,\text{Im}(\overline{H^{(1)}_\mu(r)}\frac{d}{dr}H^{(1)}_{\mu}(r))}{|H^{(1)}_\mu(r)|^2}
 =\frac{r(J_\mu(r)Y_\mu'(r)-J_\mu'(r)Y_\mu(r))}{|H^{(1)}_\mu(r)|^2}\label{eqn:ImG}.
\end{align}

We need the following estimates (see \cite{nedelec} for the case $N=3$):
\begin{lemma}
\label{sec:prop-capac-oper-2}
For $r>0$ and $\ell\geq \max\{0,\frac{3-N}{2}\}$ we have
\begin{align}
 \frac{(N-1)}{2} &\leq -\text{Re }z_\ell(r) \leq \ell + N-2 \label{estim:Re}\\
 0 &< \text{Im }z_\ell(r) \leq r.\label{estim:Im}
\end{align}
In the case $N=2$, we also have $0<-\text{Re }z_0(r)\leq \frac12$ and $\text{Im }z_0(r)>0$ for all $r>0$.  
\end{lemma}

\begin{proof}
We first prove \eqref{estim:Re}. According to \cite[\S 13.74]{watson},
we have
\begin{align*}
 (i) &\qquad \frac{d}{dr}(J_\mu^2(r)+Y_\mu^2(r))< 0 \quad \text{for $r>0, \mu\in\R$}, \\
 (ii) &\qquad \frac{d}{dr}\, r(J_\mu^2(r)+Y_\mu^2(r))\left\{\begin{array}{ll}\leq 0 & \text{if }\mu\geq\frac12 \\ \\
 \geq 0 & \text{if }\mu\leq\frac12,\end{array}\right. \quad \text{for $r>0.$}
\end{align*}
From (i) we obtain $\text{Re }G_\mu(r)<0$ for all $r>0$, $\mu \in \R$, and (ii)
gives $\text{Re }G_\mu(r)\leq -\frac12$ for $r>0$, $\mu\geq\frac12$. 
Moreover, (ii) implies $\text{Re}G_{\frac12}(r)=-\frac12$ for $r>0$. 
Now, using the recurrence formula $\frac{d}{dr}H^{(1)}_\mu(r) = H_{\mu- 1}^{(1)}(r)-\frac{\mu}{r}H_\mu^{(1)}(r)$, 
we find that
\begin{align*}
 |&H^{(1)}_\mu(r)|^2 (\text{Re }G_\mu(r)+\mu) 
 =\text{Re}\Bigl(r\overline{H^{(1)}_{\mu}(r)} \frac{d}{dr}H^{(1)}_{\mu}(r)+\mu |H^{(1)}_\mu(r)|^2\Bigr)\\
 &=r\text{Re}(\overline{H^{(1)}_{\mu}(r)}H^{(1)}_{\mu-1}(r))
 =r\text{Re}(\overline{H^{(1)}_{\mu-1}(r)}H^{(1)}_{\mu}(r))\\
 &= -|H^{(1)}_{\mu-1}(r)|^2 (\text{Re }G_{\mu-1}(r)-(\mu-1)).
\end{align*}
Since $\text{Re }G_{\mu}(r)<0$ holds for all $\mu$, the previous formula gives $\text{Re }G_\mu(r) \geq -\mu$ 
for all $\mu\geq 1$. Summarizing, we have shown
$$
 -\mu \leq \text{Re }G_\mu(r) \leq -\frac12,\quad r>0 \qquad \text{for
   $\mu\in\{{\textstyle\frac12}\}\cup[1,\infty)$.}
$$
This shows \eqref{estim:Re}, since $z_\ell(r)=G_{\ell+\frac{N-2}{2}}(r) - \frac{N-2}{2}$.
In the case $\mu=0$, we have $-\frac12\leq \text{Re }G_0(r)<0$ for all $r>0$. Hence, when $N=2$, we obtain
$-\frac12 \leq z_0(r)<0$ for all $r>0$.

We now turn to the proof of \eqref{estim:Im}. Using \eqref{eqn:ImG}
and the fact that 
\begin{equation}
  \label{eq:32}
\cW(J_\mu,Y_\mu;r):=J_\mu(r)Y_\mu'(r)-J_\mu'(r)Y_\mu(r) =\frac{2}{\pi
  r}\qquad \text{for $\mu \in \R$, $r>0$,}
\end{equation}
(see \cite[\S 3.63 (1)]{watson}), we find that 
$$
\text{Im }G_\mu(r)=\frac{r
  \cW(J_\mu,Y_\mu;r)}{|H^{(1)}_\mu(r)|^2}=\frac{2}{\pi
  |H^{(1)}_\mu(r)|^2}>0\qquad \text{for $\mu \in \R$, $r>0$.} 
$$
Furthermore, using (ii) above, we see that 
$$
r|H^{(1)}_\mu(r)|^2\geq\lim\limits_{t\to\infty}t|H^{(1)}_\mu(t)|^2=\frac2\pi
\qquad \text{for $r>0, \mu \geq \frac12$.}
$$
We thus obtain $\text{Im }G_\mu(r)\leq r$ for $r>0$, $\mu\geq\frac12$.
As before, \eqref{estim:Im} follows from the identity $z_\ell(r)=G_{\ell+\frac{N-2}{2}}(r) - \frac{N-2}{2}$.
\end{proof}

Recall that in Section~\ref{sec-some-tools} we have also introduced
the operator
\begin{equation}
  \label{eq:21}
K_R: H^{\frac12}(S_R) \to H^{-\frac12}(S_R),\qquad K_R u = \Re [T_R u],
\end{equation}
where $H^{\frac12}(S_R)=H^{\frac12}(S_R,\R)$,
$H^{-\frac12}(S_R)=H^{-\frac12}(S_R,\R)$. For $u \in
H^{\frac12}(S_R)$ given by $u(R\xi)=\sum \limits_{\ell=0}^\infty
\sum \limits_{m=1}^{d^N_\ell} u^\ell_m \cY^\ell_m(\xi)$ we then have
\begin{equation}
\label{eq:18}
 K_Ru = \frac{1}{R}\sum_{\ell=0}^\infty \text{Re}(z_\ell(kR)) \sum_{m=1}^{d^N_\ell} u^\ell_m \cY^\ell_m,
\end{equation}
since $\sum\limits_{m=1}^{d^N_\ell} u^\ell_m \cY^\ell_m$ is a real-valued
function for every
$\ell \in \N_0$. Using Lemma~\ref{sec:prop-capac-oper-2}
we may therefore easily complete the

\begin{proof}[Proof of Lemma~\ref{sec:capacity-operator}]
We first note that, for $s \in \R$, one may define an equivalent norm
on $H^s(S_R,\C)$ by 
$$
\|u\|_s^2 = \sum_{\ell=0}^\infty (1+\ell^2)^s \sum_{m=1}^{d^N_\ell}
|u^\ell_m|^2
$$
for $u \in H^{\frac12}(S_R,\C)$ given by $u(R\xi)=\sum \limits_{\ell=0}^\infty
\sum \limits_{m=1}^{d^N_\ell} u^\ell_m \cY^\ell_m(\xi)$. Hence, by
\eqref{estim:Re} and \eqref{estim:Im}, the capacity operator $T_R:
H^{\frac12}(S_R,\C) \to H^{-\frac12}(S_R,\C)$ is bounded, and thus
$K_R$ defined in \eqref{eq:21} is bounded as well. In particular, the
integrals on both sides of \eqref{eq:16} exist, and by \eqref{eq:18} and the
orthogonality of spherical harmonics (corresponding to different
values of $\ell$) they coincide. Finally, the second inequality in
\eqref{eq:5} also follows, by orthogonality, from \eqref{eq:18} and
the estimates in Lemma~\ref{sec:prop-capac-oper-2}.
\end{proof}

Next, we study the eigenspace $X^0$ of the eigenvalue
problem \eqref{eq:10} corresponding to the eigenvalue $k^2$, and we complete the 

\begin{proof}[Proof of Lemma~\ref{sec:some-tools-2}(iv)]
Recall that  $u \in X^0$ if and only if $u$ solves 
\begin{equation}
\label{eq:12}
\left\{
  \begin{aligned}
\Delta u + k^2 u &= 0 &&\qquad \text{in $B_R$,}\\
\frac{\partial u}{\partial\eta}&=v &&\qquad \text{on $S_R$}  
  \end{aligned}
\right.
\end{equation}
with $v=K_R u$ on $S_R$. Put, as before, $\mu_\ell= \ell + \frac{N-2}{2}$ for $\ell \in \N_0$. 
In the case where
\mbox{$\frac{d}{dr}\left(r^{\frac{2-N}{2}}J_{\mu_\ell}(kr)\right)\Big
  |_{r=R}\not= 0$} for all $\ell\in\N_0$, 
the inhomogeneous Neumann problem \eqref{eq:12} has, for given 
$v \in  H^{-\frac{1}{2}}(S_R)$, the unique solution 
$$
u(r\xi) = \left(\frac{r}{R}\right)^{-\frac{N-2}{2}}\sum_{\ell=0}^\infty
\frac{J_{\mu_\ell}(kr)}{kJ'_{\mu_\ell}(kR)
  - \frac{(N-2)}{2R}J_{\mu_\ell}(kR)} \sum_{m=1}^{d^N_\ell}v^\ell_m
\cY^\ell_m(\xi),  \qquad 0<r<R,
$$
where the coefficients $v^\ell_m$ are determined by $v(R\xi)=
\sum\limits_{\ell=0}^\infty\sum\limits_{m=1}^{d^N_\ell} v^\ell_m
\cY^\ell_m(\xi)$. In particular, the restriction of $u$ to $S_R$
satisfies 
$$
u(R\xi) = \sum_{\ell=0}^\infty\sum_{m=1}^{d^N_\ell}u^\ell_m
\cY^\ell_m(\xi)  \qquad \text{with} \;
u^\ell_m = \frac{J_{\mu_\ell}(kR)}{kJ'_{\mu_\ell}(kR) - \frac{(N-2)}{2R}J_{\mu_\ell}(kR)} v^\ell_m.
$$
If, in addition, $v= K_Ru$ on $S_R$, then by \eqref{eq:18} we have 
\begin{align*}
v^\ell_m&= \frac{1}{R}\text{Re} \left(\frac{kR\frac{d}{dr}H^{(1)}_{\mu_\ell}(kR)}{H^{(1)}_{\mu_\ell}(kR)}
-\frac{N-2}{2} \right) u^\ell_m\\
&=k\left(\frac{J'_{\mu_\ell}(kR)J_{\mu_\ell}(kR)+Y'_{\mu_\ell}(kR)Y_{\mu_\ell}(kR)}{J^2_{\mu_\ell}(kR)
+Y^2_{\mu_\ell}(kR)}-\frac{N-2}{2kR}\right) u^\ell_m
\end{align*}
and therefore  
$$
u^\ell_m = \frac{J_{\mu_\ell}(kR)}{J'_{\mu_\ell}(kR)
  - \frac{(N-2)}{2k R}J_{\mu_\ell}(kR)}\left(\frac{J'_{\mu_\ell}(kR)J_{\mu_\ell}(kR)
  +Y'_{\mu_\ell}(kR)Y_{\mu_\ell}(kR)}{J^2_{\mu_\ell}(kR)+Y^2_{\mu_\ell}(kR)}
-\frac{N-2}{2kR}\right)u^\ell_m
$$
for $\ell\in\N_0$, $1\leq m\leq d^N_\ell$. As a consequence of
(\ref{eq:32}), this gives for each $\ell\in\N_0$ the alternative
$$
u^\ell_m=0\quad\text{for all }1\leq m\leq d^N_\ell \quad\text{ or }\quad Y_{\mu_\ell}(kR)=0.
$$
We may now finish the proof by setting 
$$
\cD:= \Bigl\{R >0 \::\: \text{$\frac{d}{dr}\left(r^{\frac{2-N}{2}}J_{\mu_\ell}(kr)\right)\Big
|_{r=R}=0$ or $Y_{\mu_\ell}(kR) = 0$ for some $\ell\in\N_0$}\Bigr\}
$$
Indeed, since $J_{\mu_\ell}$ and $Y_{\mu_\ell}$ are analytic functions
on $(0,\infty)$ for every $\ell \in \N_0$, the set $\cD$ is
countable. Moreover, for
$R \in (0,\infty) \setminus \cD$ and $u \in X^0$ we conclude
$u^\ell_m= 0$ for all $\ell \in \N_0$, $1 \le m \le d_\ell^N$ and
therefore $u \equiv 0$.
\end{proof}

We close this section with the 

\begin{proof}[Proof of Theorem~\ref{thm:1-1}]
Let $R>0$ be fixed, and let $u \in \cR_R$ with $u|_{S_R} \not \equiv
0$. Furthermore, let $R'>R>0$ be given such that $u \in
\cR_{R'}$. Then 
we have: 
\begin{itemize}
\item[(i)] In $E_R$, $u$ coincides with the real
part of a solution $w_1$ of the linear Helmholtz equation satisfying~\eqref{eqn:radiation1};
\item[(ii)] In $E_{R'}$, $u$ coincides with the real
part of a solution $w_2$ of the linear Helmholtz equation satisfying~\eqref{eqn:radiation1};
\item[(iii)] $u|_{S_R} \equiv w_1|_{S_R}$ and $u|_{S_{R'}} \equiv w_2|_{S_{R'}}$.
\end{itemize}
Since $w:=w_1-w_2$ also
satisfies the linear Helmholtz equation in $E_{R'}$ together with
\eqref{eqn:radiation1} and $\Re(w) \equiv 0$ in $E_{R'}$, it follows
that $w(x)= o(r^{\frac{1-N}{2}})$ as $r=|x| \to \infty$ and hence $w
\equiv 0$ in $E_{R'}$ by a classical result of Rellich~\cite[Satz 1]{rellich43}.
Hence $w_1 = w_2$ in $E_{R'}$, and by (iii) this implies that 
$\text{Im}(w_1|_{S_{R'}})\equiv \text{Im}(w_2|_{S_{R'}})\equiv 0$. Writing again 
\begin{equation}
  \label{eq:3}
u(R \xi) = w_1(R\xi) =\sum_{\ell=0}^\infty
u^{\ell}(\xi),\qquad \xi \in S_1,
\end{equation}
with $u^{\ell}(\xi) = \sum \limits_{m=1}^{d^N_\ell} u^\ell_m
\cY^\ell_m(\xi)$, we have, by Proposition~\ref{sec:prop-capac-oper-1},
$$
w_1(R' \xi)=\left(\frac{R'}{R}\right)^{-\frac{(N-2)}{2}} \sum_{\ell=0}^\infty  
 \frac{H^{(1)}_{\mu_\ell}(kR')}{H^{(1)}_{\mu_\ell}(kR)} 
 u^{\ell}(\xi), 
  \qquad \xi \in S_1,\\
$$
where the functions $\xi \mapsto u^{\ell}(\xi)$, $\ell \in \N_0$ are real-valued as a consequence of (\ref{eq:3}) and
the fact that $u$ is real-valued. Moreover, the assumption $u|_{S_R}
\not \equiv 0$ forces that $u^\ell \not \equiv 0$ on $S_1$ for at least
one $\ell \in \N_0$, and therefore 
 $$
\text{Im}\Bigl(\frac{H^{(1)}_{\mu_\ell}(kR')}{H^{(1)}_{\mu_\ell}(kR)}\Bigr)
=0, \qquad \text{i.e.,} \qquad
J_{\mu_\ell}(kR')Y_{\mu_\ell}(kR)=J_{\mu_\ell}(kR)Y_{\mu_\ell}(kR').
$$
Since the functions $J_{\mu_\ell}$ and $Y_{\mu_\ell}$ are real
analytic, the latter can only happen for countably many
$R'>R$. We thus conclude that $u \in \cR_{R'}$ for at most countably
many $R'>R$.\\
It remains to be shown that problem
(\ref{eq:33}),~(\ref{eqn:radiation2}) admits infinitely many solutions
if the nonlinearity $f$ satisfies $(f0)$--$(f4)$. For this we fix $R_1
\in (0,\infty) \setminus \cD$ such that $\Omega \subset B_{R_1}$; then Theorem~\ref{thm:1}
yields a nontrivial solution $u_1 \in \cR_{R_1}$ of
(\ref{eq:33}),~(\ref{eqn:radiation2}). The nontriviality of $u$
implies that $u_1|_{S_{R_1}} \not \equiv 0$, since otherwise $u \equiv
  0$ on $E_{R_1}$ and therefore on all of $\R^N$ as a consequence of
  unique continuation. Hence there exists
$R_2 \in (R_1,\infty) \setminus \cD$ such that $u_1 \not \in \cR_{R_2}$, whereas 
Theorem~\ref{thm:1} yields a solution $u_2 \in \cR_{R_2}$ of
(\ref{eq:33}),~(\ref{eqn:radiation2}). Moreover, there exists
$R_3 \in (R_2,\infty) \setminus \cD$ such that $u_1,u_2 \not \in \cR_{R_2}$, whereas Theorem~\ref{thm:1} yields a solution $u_3 \in \cR_{R_3}$ of
(\ref{eq:33}),~(\ref{eqn:radiation2}). Inductively, we now obtain an infinite
sequence of pairwise different solutions of (\ref{eq:33}),~(\ref{eqn:radiation2}).
\end{proof}

\end{document}